\documentclass{article}

\usepackage{microtype}
\usepackage{graphicx}
\usepackage{multirow}
\usepackage{subcaption}
\usepackage{booktabs} 
\usepackage{subcaption}
\usepackage{threeparttable}
\usepackage{hyperref}

\usepackage[preprint]{icml2026}

\usepackage{amsmath}
\usepackage{amssymb}
\usepackage{mathtools}
\usepackage{amsthm}

\usepackage{multirow}
\usepackage{graphicx}
\usepackage{array}

\usepackage[capitalize,noabbrev]{cleveref}

\theoremstyle{plain}
\newtheorem{theorem}{Theorem}[section]

\newtheorem{lemma}[theorem]{Lemma}
\newtheorem{corollary}[theorem]{Corollary}
\theoremstyle{definition}

\newtheorem{assumption}[theorem]{Assumption}
\theoremstyle{remark}

\usepackage[textsize=tiny]{todonotes}
\icmltitlerunning{DePCoN for Robust Parameter Estimation in Non-autonomous System with Discontinuous Inputs}

\begin{document}
\twocolumn[
  \icmltitle{Deep Predictor-Corrector Networks for Robust Parameter Estimation in Non-autonomous System with Discontinuous Inputs}



  \icmlsetsymbol{equal}{*}
  \icmlsetsymbol{corre}{$\dagger$}

  \begin{icmlauthorlist}
    \icmlauthor{Gyeongwan Gu}{ajou,equal}
    \icmlauthor{Jinwoo Hyun}{ibs,equal}
    \icmlauthor{Hyeontae Jo}{ajou,ibs,corre}
    \icmlauthor{Jae Kyoung Kim}{ibs,kaist,korea,corre}
  \end{icmlauthorlist}

  \icmlaffiliation{ajou}{Department of Mathematics, Ajou University, Suwon 16499, Republic of Korea}
  \icmlaffiliation{ibs}{Biomedical Mathematics Group, Institute for Basic Science, Daejeon 34126, Republic of Korea}
  \icmlaffiliation{kaist}{Deparment of Mathematical Sciences, KAIST, Daejeon 34141, Republic of Korea }
  \icmlaffiliation{korea}{Department of Medicine, College of Medicine, Korea University, Seoul 02841, Republic of Korea}

  \icmlcorrespondingauthor{Hyeontae Jo}{ajouhtj@ajou.ac.kr}
  \icmlcorrespondingauthor{Jae Kyoung Kim}{jaekkim@kaist.ac.kr}

  \icmlkeywords{Non-smooth optimization, Non-autonomous differential equations, Discontinuous inputs, Parameter estimation, Kernel smoothing, Deep learning}

  \vskip 0.3in
]




\printAffiliationsAndNotice{}  


\begin{abstract}
Learning under non-smooth objectives remains a fundamental challenge in machine learning, as abrupt changes in conditioning variables can induce highly non-smooth loss landscapes and destabilize optimization.
This difficulty is particularly pronounced in non-autonomous dynamical systems driven by discontinuous inputs, where widely used optimization methods, including recent neural smoothing approaches, exhibit unreliable convergence or strong hyperparameter sensitivity.
To address this issue, we propose Deep Predictor--Corrector Networks (DePCoN), a multi-scale learning framework that stabilizes optimization by learning scale-consistent parameter update rules across a hierarchy of smoothed inputs.
Rather than treating smoothing as a fixed preprocessing choice, DePCoN integrates smoothing into the learning dynamics itself through a learned predictor--corrector mechanism.
Across biological and ecological benchmarks with discontinuous inputs, DePCoN consistently achieves more robust and faster convergence than existing methods while substantially reducing sensitivity to hyperparameter choices.
Beyond dynamical systems, our approach provides a general learning principle for stabilizing optimization under non-smooth objectives.
\end{abstract}

\section{Introduction}
Many machine learning models fail not only because of insufficient expressiveness, but also because of ill-posed optimization problem under non-smooth objectives \cite{Goffin1977, Nesterov2005}.
In many practical settings, abrupt changes in inputs or conditioning variables can induce highly non-smooth loss landscapes, rendering widely used optimization methods unstable and unreliable.
This challenge arises across a wide range of learning problems and persists even in low-dimensional settings.

\begin{table}[t!]
\caption{Representative applications of non-autonomous dynamical systems driven by discontinuous exogenous inputs.}
\label{tab:nonauto_appl_summary}
\centering
\resizebox{\linewidth}{!}{%
\begin{tabular}{lll}
\toprule
Category & Application & \begin{tabular}[c]{@{}c@{}}Input type (Reference)\end{tabular} \\
\midrule
Control & Traffic signals & \begin{tabular}[l]{@{}l@{}}Periodic red\/green switching\\ \cite{33aboudolas2009store}\end{tabular}
\\ \cmidrule(rl){1-3}
Bio/Life & \begin{tabular}[l]{@{}l@{}}NF-$\kappa$B signaling \\ pathways\end{tabular} & \begin{tabular}[l]{@{}l@{}}External stimuli (e.g. TNF-$\alpha$) \\ \cite{41nelson2004oscillations}\end{tabular} \\ \cmidrule(rl){1-3}
\begin{tabular}[l]{@{}l@{}}Electronic\\engineering\end{tabular}  & Power electronics & \begin{tabular}[l]{@{}l@{}}Component faults \\\cite{22dominguez2010detection}\end{tabular}\\ \cmidrule(rl){1-3}
Finance & Financial control & \begin{tabular}[l]{@{}l@{}}Central bank interventions \\\cite{19cadenillas2000classical}\end{tabular}\\ \cmidrule(rl){1-3}
Neural networks & Neural dynamics & \begin{tabular}[l]{@{}l@{}}Spiking synaptic activity \\\cite{10stamov2007almost}\end{tabular} \\
\cmidrule(rl){1-3}
Epidemiology & \begin{tabular}[l]{@{}l@{}}Non-pharmaceutical\\interventions\end{tabular} & \begin{tabular}[c]{@{}c@{}}Discrete policy shocks\\\cite{40flaxman2020estimating}\end{tabular} \\
\bottomrule
\end{tabular}}
\vskip -0.2in
\end{table}

A prominent instance of this challenge arises in non-autonomous differential equations, which are widely used to describe natural and engineered systems driven by time-varying exogenous inputs~\cite{g1aastrom2021feedback, g2Coppel1978Dichotomies}.
Such models support a broad range of tasks, including prediction, control, system design, optimization, and the evaluation of interventions or policy scenarios across scientific and engineering domains.
In many real-world systems, these exogenous inputs are inherently non-smooth, as they are generated by switching rules, pulsed actions, or event-triggered decisions.
This setting is encountered, for example, in control systems with red--green traffic-light switching or event-based demand response~\cite{33aboudolas2009store,42du2014positive,34mathieu2012using,36chassin2015new}, in power electronics with high-frequency on--off switching~\cite{43papafotiou2004hybrid,45siddhartha2018non}, and in biological and circadian models involving bolus dosing, abrupt external stimulation, or light--dark forcing~\cite{15huang2012modeling,41nelson2004oscillations,39forger1999simpler,song2023real,44lim2025enhanced} (See \cref{tab:nonauto_appl_summary} and \cref{tab:nonauto_appl_full} for representative examples).

Recent work has attempted to mitigate this issue by smoothing discontinuous inputs using kernels or neural network approximations and performing parameter estimation on the resulting surrogate dynamics~\cite{beck2012smoothing, yang2014mollification, lin2023continuation, ko2023homotopy, xu2024global}.
While such approaches can improve numerical stability in certain cases, they still rely on selecting a single smoothing scale or network configuration, making performance highly sensitive to the hyperparameter choice: insufficient smoothing fails to stabilize optimization, whereas excessive smoothing distorts the underlying dynamics and biases the recovered parameters.
As a consequence, existing smoothing-based methods often trade optimization stability for estimation bias rather than resolving the core optimization difficulty.
\begin{figure*}[ht!]
  \begin{center}
    \centerline{\includegraphics[width=0.85\linewidth]{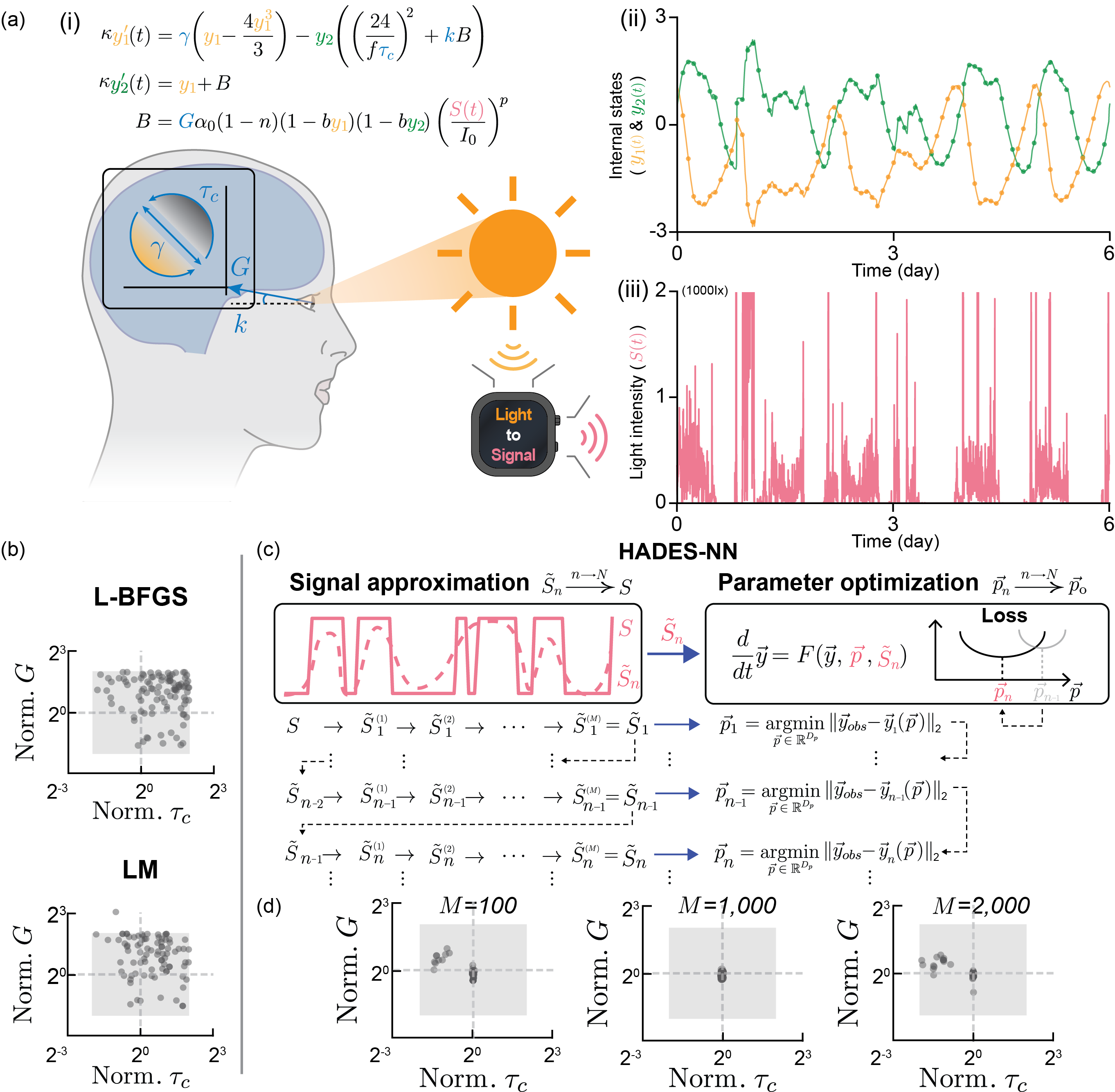}}
    \caption{Failure of widely used local optimization under discontinuous inputs and hyperparameter sensitivity of HADES-NN.
(a) Human circadian pacemaker model in which discontinuous light exposure $S(t)$ is transformed into an effective photic drive $B(t)$ and subsequently modulates the internal oscillator states $\vec y(t)=\left(y_1(t),y_2(t)\right)$ (i). Synthetic observations are generated (ii) by integrating the system for 144 hours using real-world light measurements $S(t)$ (iii).
(b) Scatter plots of parameter estimates obtained by widely used optimization methods, including L-BFGS and LM, from randomized initializations. The gray box indicates the range of parameters from which initial values were randomly sampled. All estimates are normalized by the true parameter values, which are located at the intersection of the dashed lines. Each point represents the final estimate from 30 independent optimization trials, illustrating large dispersion and unstable convergence under irregular exogenous inputs.
(c) Schematic of HADES-NN, which alternates between neural smoothing of the discontinuous input and parameter optimization. At each outer iteration, a smoothed input $\tilde S_n$ is refined through $M$ inner updates and used to estimate system parameters.
(d) Sensitivity of HADES-NN to the smoothing depth $M$. While increasing $M$ enhances regularization, excessive smoothing suppresses informative input variations and induces estimation bias, whereas small $M$ provides insufficient stabilization.}
    \label{fig:1}
  \end{center}
  \vskip -0.3in
\end{figure*}

In this paper, we introduce Deep Predictor--Corrector Networks (DePCoN), a multi-scale learning framework that stabilizes optimization under non-smooth objectives by learning scale-consistent update rules in parameter space. Rather than treating smoothing as a fixed preprocessing choice, DePCoN embeds smoothing into the learning dynamics itself by jointly leveraging a hierarchy of smoothed inputs.
A learned predictor propagates parameter estimates across smoothing scales, while a corrector enforces data consistency at each scale through Neural ODE–based system identification. By aggregating estimation errors across scales, DePCoN transforms smoothing from a fragile hyperparameter into a structured learning input. We further provide a convergence analysis showing that, as the smoothing scale vanishes, the resulting sequence of optimal parameter estimates converges to the global minimizer under standard stability assumptions (Appendix \ref{convergence}).

Through experiments on biological and ecological systems with discontinuous inputs, we show that DePCoN consistently achieves stable and accurate parameter estimation in regimes where widely used optimization and existing neural smoothing approaches fail. Beyond non-autonomous dynamical systems, our framework provides a general learning principle for stabilizing optimization under non-smooth objectives, with potential applications to a broad class of machine learning problems involving discontinuous conditioning or highly non-smooth loss landscapes.

\section{Related Works}
A common strategy to mitigate optimization difficulties induced by non-smooth objectives is to replace the original problem with a smooth surrogate that is easier to optimize \cite{beck2012smoothing}.
One approach is kernel-based smoothing, in which the input signal is regularized via convolution with a smoothing kernel.
This construction follows a well-established principle in mathematical analysis, where convolution with an approximate identity yields smooth approximations that converge to the original function as the smoothing scale decreases (see, e.g., \cite{friedrichs1944identity,hwang2019diffusive} and Chapter~5 of \cite{evans2022partial}).
Closely related ideas have also appeared in modern deep learning as continuation or homotopy-based optimization heuristics, where objective smoothing is used to reduce the ruggedness the loss landscape and stabilize training \cite{lin2023continuation,xu2025global}.
In practice, however, the choice of the smoothing scale is critical: insufficient smoothing may fail to eliminate non-smoothness and lead to numerical instability, whereas excessive smoothing can distort the original dynamics and introduce systematic estimation bias.

More recently, neural-network-based smoothing approaches have been proposed, in which discontinuous exogenous inputs are approximated using neural networks with smooth activation functions, thereby inducing a smooth surrogate objective for parameter estimation.
A representative example is Harmonic Approximation of Discontinuous External Signals using Neural Networks (HADES-NN)~\cite{doi:10.1137/25M1741340}, which alternates between neural approximation of the input signal and parameter optimization on the resulting surrogate dynamics.
While universal approximation results guarantee that such networks can approximate the target input in principle, the effective smoothness of the learned representation depends strongly on architectural and optimization hyperparameters, such as network depth, width, activation functions, and learning rates.
As a consequence, neural-network-based smoothing such as HADES-NN can exhibit hyperparameter-dependent bias, producing surrogate dynamics that deviate from the original system even when the optimization procedure converges.

A common limitation shared by both kernel-based and neural-network-based smoothing approaches is that they rely on a single, implicitly chosen smoothing scale or effective smoothness level, making the resulting optimization sensitive to hyperparameter choices.
Moreover, smoothing is typically treated as an auxiliary or preprocessing step, rather than being integrated into the learning dynamics of the parameter estimation process itself.

\section{Main Results}
\subsection{Failure of Widely Used Optimizers under Discontinuous Forcing}
We begin by revisiting a representative failure mode of parameter estimation in non-autonomous systems driven by discontinuous real-world inputs, as formulated in ~\cref{de}, using the human circadian pacemaker model as an illustrative testbed (\cref{fig:1}-(a) i)).
This model was already employed in previous work~\cite{doi:10.1137/25M1741340} to demonstrate that widely used optimization methods, including deep-learning-based approaches, can yield unstable and highly dispersed parameter estimates under irregular or abruptly switching exogenous signals.

The circadian pacemaker model is a mechanistic non-autonomous oscillator in which the wearable-derived light signal $S(t)$ is transformed into an effective photic drive $B(t)$ and subsequently coupled to the internal states $\vec{y}(t)=(y_1(t),y_2(t))$.
Inter-individual variability in circadian responses is summarized by the set of four parameters $\vec{p}=(p_1,p_2,p_3,p_4)=(\tau_c,\gamma,G,k)$, representing the intrinsic period, oscillator stiffness, photic gain, and coupling strength, respectively (see \cite{doi:10.1137/25M1741340} for details).
We numerically integrate the model over 144 hours (6 days) using the measured light exposure $S(t)$ as input and construct synthetic observation data by uniformly sampling 80 time points,
$\{\vec{y}_{\mathrm{o}}(t_i)\}_{i=1}^{80}$ (\cref{fig:1}-(a) ii)-iii)).
\begin{figure*}[t!]
  \begin{center}
    \centerline{\includegraphics[width=0.85\linewidth]{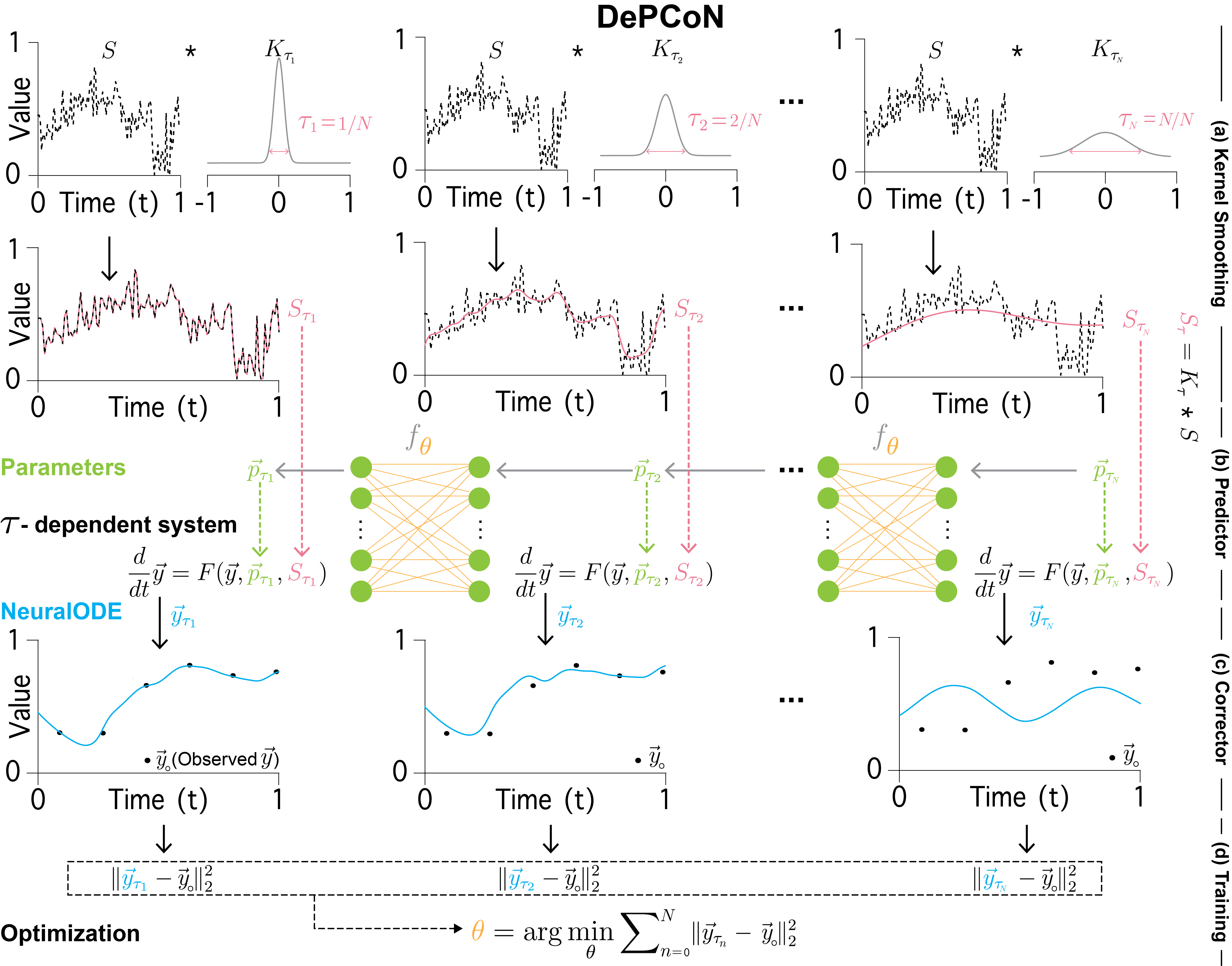}}
    \caption{DePCoN stabilizes parameter estimation by learning parameter update rules that remain consistent across multiple input-smoothing scales.
(a) Multi-scale preprocessing of the discontinuous exogenous input $S(t)$ via heat-kernel convolution, producing a hierarchy of smoothed inputs $\{S_{\tau_n}(t)\}_{n=0}^{N}$ with normalized scales $\tau_n=n/N$, from lightly smoothed ($\tau_1$) to heavily smoothed ($\tau_N$).
(b) Predictor stage. Starting from the coarsest scale $\tau_N$, the predictor network $f_\theta$ propagates parameter estimates toward finer scales according to $\vec p_{\tau_{n-1}} = f_\theta(\vec p_{\tau_n})$, transferring stable information to increasingly discontinuous regimes.
(c) Corrector stage. For each scale $\tau_n$, the predicted parameters $\vec p_{\tau_n}$ and smoothed input $S_{\tau_n}(t)$ define a $\tau$-dependent non-autonomous system whose Neural ODE solution $\vec y_{\tau_n}$ is compared with observations.
(d) Training objective. The predictor network is trained end-to-end by minimizing a multi-scale loss $\sum_{n=0}^{N}\|\vec y_{\tau_n}-\vec y_{\mathrm{o}}\|$, enforcing consistency across scales and reducing sensitivity to any single smoothing choice.
}
    \label{fig:2}
  \end{center}
\vskip -0.3in
\end{figure*}
Given these observations, we estimate $\vec{p}$ that minimize a loss function~\cref{L}.
To assess robustness with respect to initialization, parameter estimation is repeated with random initial guesses sampled independently from $[p_{\mathrm{o},i}/4,\,4p_{\mathrm{o},i}]$ for $i=1,\ldots,4$ (gray boxes in \cref{fig:1}-(b)).
 Both Limited-memory Broyden–Fletcher–Goldfarb–Shanno (L-BFGS) and Levenberg--Marquardt (LM), despite being standard choices in practice, produce widely dispersed estimates across runs rather than consistently converging to the ground-truth parameters, indicated by the horizontal dashed line at the reference level $2^0 = 1$ (\cref{fig:1}-(b),\cref{fig:5}-(a)). This lack of convergence is not limited to gradient-based methods. As shown in \cref{fig:5}-(a), several widely used derivative-free or constrained optimization algorithms—including Differential Evolution (DE), Nelder–Mead (NM), and Sequential Least Squares Quadratic Programming (SLSQP)—also fail to recover meaningful parameter estimates under the same discontinuous-input setting. In particular, DE frequently diverges numerically, while NM and SLSQP remain trapped near their initializations, yielding highly dispersed or systematically biased estimates.
 
These results are consistent with the failure mode reported in \cite{doi:10.1137/25M1741340} and reflects a fundamental difficulty: discontinuous and irregular inputs induce highly non-smooth loss landscapes that degrade the performance of widely used optimization methods.

\subsection{Hyperparameter Sensitivity of HADES-NN}\label{hyper}
\noindent
Recently, HADES-NN was proposed to mitigate this issue by learning smooth approximations of discontinuous inputs $S(t)$ and performing parameter optimization using the resulting dynamics \cite{doi:10.1137/25M1741340}.
HADES-NN alternates between signal smoothing and parameter optimization over an outer loop indexed by $n=1,\ldots,N$ (\cref{fig:1}-(c)). At the $n$-th outer iteration, a neural-network-based smoother refines a smoothed approximation of the discontinuous input through $M$ inner updates $(m=1,\ldots,M)$, yielding a progressively refined signal $\tilde S_n^{(M)}$, denoted by $\tilde{S}_n$.
Using this refined input $\tilde{S}_n$, HADES-NN integrates the corresponding dynamical system~\cref{de} by replacing $S(t)$ with $\tilde{S}_n$ and updates the parameter estimate by fitting simulated trajectories to the observations, yielding an optimized parameter vector $\vec{p}_n$.
Crucially, successive outer iterations are coupled through a continuation mechanism: the refined input approximation $\tilde{S}_n$ is used to initialize the next smoothing stage, while the estimated parameters $\vec{p}_n$ serve as the initial condition for the subsequent parameter optimization.

To assess the sensitivity of HADES-NN to the smoothing depth, we evaluate parameter estimates obtained with different values of the inner-loop iteration number ($M=100,\,1000,$ and $2000$, while fixing $N=200$) (\cref{fig:1}-(d)).
Although HADES-NN improves estimation accuracy over widely used optimization baselines (\cref{fig:1}-(b)), the results reveal a pronounced dependence on $M$.
When $M$ is small, insufficient regularization leads to increased dispersion of parameter estimates, whereas large $M$ can over-smooth informative input variations and introduce bias.

In addition to accuracy sensitivity, increasing $M$ incurs a non-negligible computational cost.
Because each smoothing update requires repeated neural approximation and numerical integration, larger values of $M$ substantially increase runtime.
Together, these results indicate that the optimal choice of $M$ is not obvious \emph{a priori} and requires careful tuning to balance accuracy, robustness, and computational efficiency.

\subsection{DePCoN: A Predictor--Corrector Framework for Parameter Estimation under Discontinuous Inputs}

To overcome the hyperparameter sensitivity and computational burden observed in existing approaches, we propose \emph{Deep Predictor--Corrector Networks (DePCoN)}, a parameter-estimation framework that integrates information across multiple smoothing scales in a principled manner.
Rather than relying on a single smoothed input, DePCoN exploits a hierarchy of smoothed signals and learns scale-consistent parameter updates, leading to stable estimation under discontinuous exogenous inputs.

As illustrated in \cref{fig:2}, DePCoN consists of a preprocessing step followed by two coupled stages: a predictor stage and a corrector stage.
In the preprocessing step (\cref{fig:2}-(a)), we apply heat-kernel convolution to the discontinuous input signal $S(t)$ to generate a family of smoothed signals $\{S_{\tau_n}(t)\}_{n=0}^{N}$~\cref{smoothing_input}.
The smoothing scale is normalized and discretized as $\tau_n = n/N$ for $n=0,1,\ldots,N$, producing a hierarchy of inputs ranging from unsmoothed ($\tau_0=0$), through lightly smoothed ($\tau_1=1/N$), to heavily smoothed ($\tau_N=1$) (See \cref{subsec:smoothing} for the mathematical interpretation of the case $\tau_0=0$).

The predictor stage (\cref{fig:2}-(b)) learns to propagate parameter estimates across smoothing scales.
Starting from the coarsest scale $\tau_N$, the predictor network $f_\theta$ sequentially generates parameter estimates along the scale hierarchy according to $\vec p_{\tau_{n-1}} = f_\theta(\vec p_{\tau_n})$ for $n=1,\ldots,N$ (see architectural details in \cref{subsec:pred_net}).
Through this process, the predictor transfers stable parameter information obtained under heavy smoothing to finer scales where the effects of discontinuities are more pronounced.

\begin{figure}[t!]
  \vskip 0.2in
  \begin{center}
    \centerline{\includegraphics[width=\linewidth]{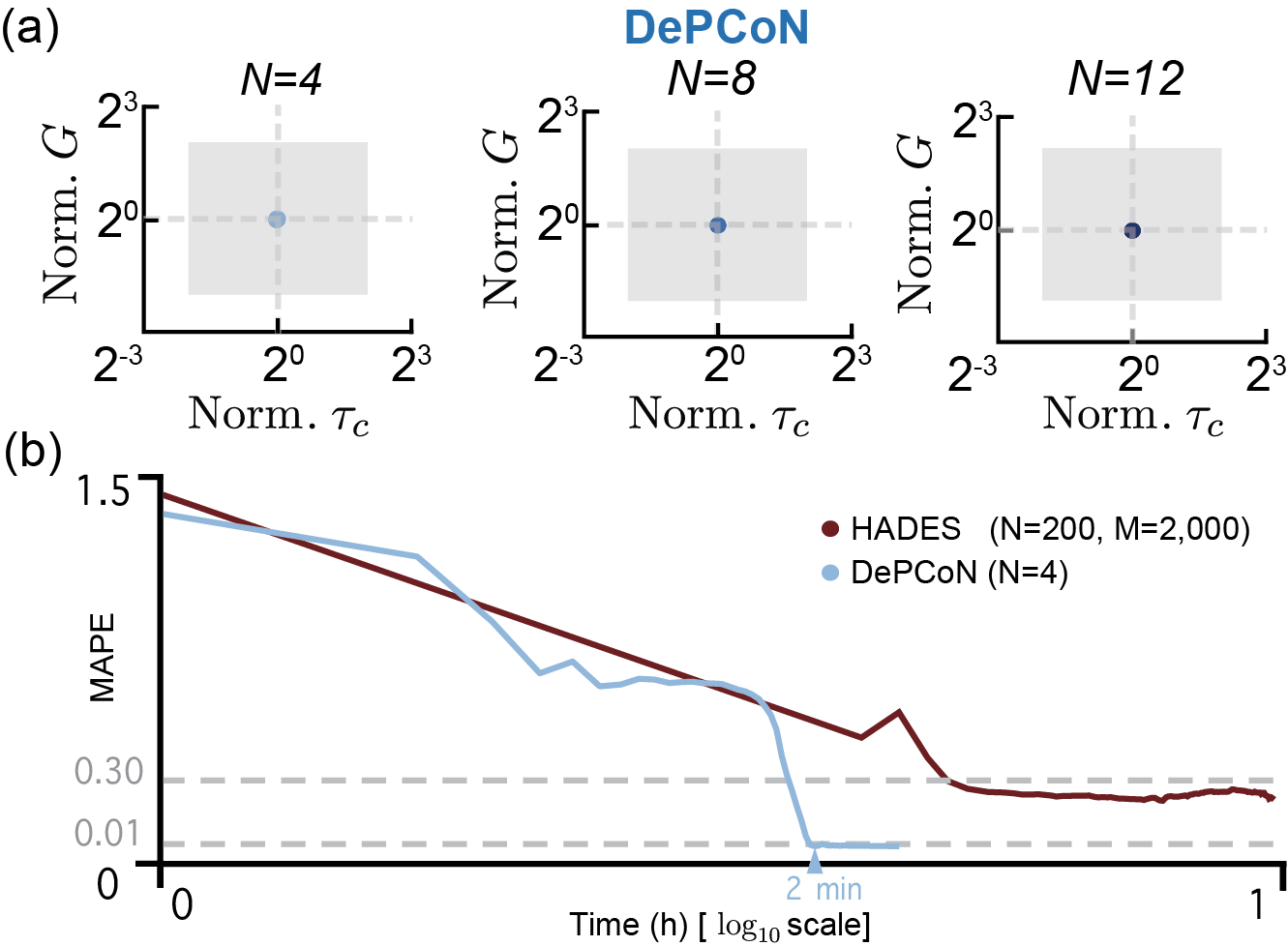}}
    \caption{Hyperparameter robustness and computational efficiency of DePCoN.
(a) Parameter estimation results of DePCoN under different smoothing-grid sizes $N=\{4,8,12\}$.
Each panel shows the distribution of estimated parameters $(\tau_c, G)$ obtained from randomized initializations.
In contrast to the pronounced hyperparameter sensitivity observed for HADES-NN in Fig.~1(d), the estimates produced by DePCoN remain tightly concentrated across all values of $N$, indicating substantially reduced sensitivity to the smoothing-grid hyperparameter. Gray boxes indicate the parameter ranges used for random initialization, and all estimates are normalized by the true parameter values, located at the intersection of the dashed lines.
(b) MAPE trajectories over a 1-hour interval comparing DePCoN with HADES-NN.
DePCoN rapidly enters a low-error regime and exhibits stable convergence,
whereas HADES-NN converges more slowly and remains at a higher error level.}
    \label{fig:3}
  \end{center}
\vskip -0.3in
\end{figure}

In the corrector stage (\cref{fig:2}-(c)), the predicted parameters $\{\vec p_{\tau_n}\}_{n=0}^{N}$ are evaluated against the observations across all smoothing scales.
For each scale $\tau_n$, the smoothed input $S_{\tau_n}(t)$ together with the corresponding parameter estimate $\vec{p}_{\tau_n}$ defines a $\tau$-dependent non-autonomous dynamical system~\cref{de_tau}, whose Neural ODE solution is denoted by $\vec{y}_{\tau_n}$.
The discrepancy between $\vec y_{\tau_n}$ and the observed data $\vec y_{\mathrm{o}}$ provides a scale-specific estimation error (See \cref{subsec:corr_net} for the detailed formulation).

\begin{figure*}[t!]
  \begin{center}
    \centerline{\includegraphics[width=0.95\linewidth]{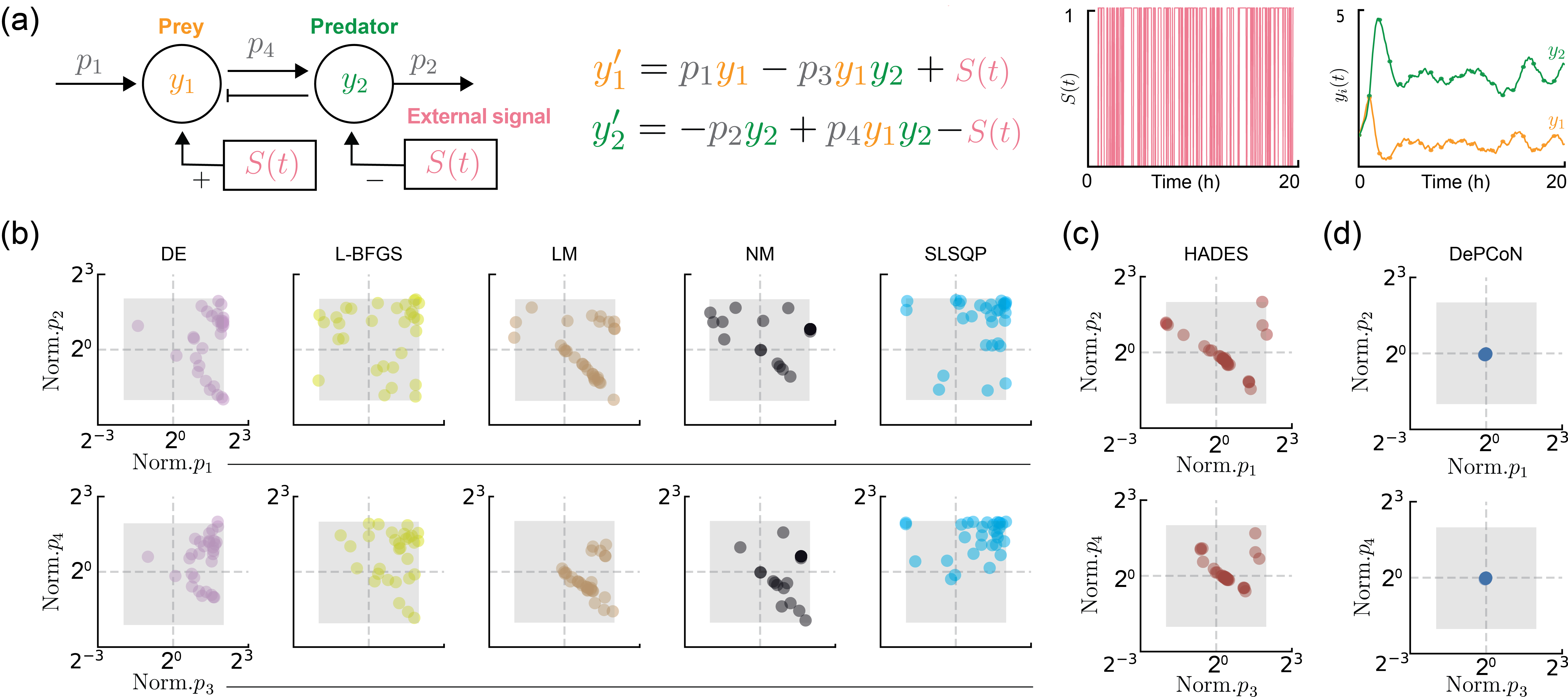}}
    \caption{
Robust parameter estimation on the Lotka--Volterra system with discontinuous exogenous inputs.
(a) Modified Lotka--Volterra model describing prey--predator dynamics driven by an abrupt external signal $S(t)$.The resulting trajectories exhibit pronounced non-smooth behavior under real-world--like environmental perturbations.
(b) Parameter estimates obtained using widely used optimization methods under identical experimental settings. Each scatter plot shows estimates of parameter pairs from 30 independent trials with randomized initializations, revealing broad dispersion and strong sensitivity to initialization.
(c) Results for HADES-NN using hyperparameter configurations reported to perform well in prior studies~\cite{doi:10.1137/25M1741340}. Despite neural smoothing of the discontinuous input, the resulting estimates remain unstable and exhibit systematic bias away from the ground-truth parameters.
(d) DePCoN yields tightly clustered parameter estimates centered at the ground truth, demonstrating both high accuracy and strong consistency.  
All parameter estimates are normalized by the true parameter values, which are located at the intersection of the dashed lines, and gray boxes indicate the parameter ranges used for random initialization.}
    \label{fig:4}
  \end{center}
\vskip -0.3in
\end{figure*}

Crucially, the predictor network $f_\theta$ is trained end-to-end by minimizing the aggregated estimation error across all scales, i.e., by minimizing $\sum_{n=0}^{N} \|\vec y_{\tau_n} - \vec y_{\mathrm{o}}\|$~\cref{L_tau}-\cref{L_total}.
This multi-scale loss enforces consistency of parameter estimates across smoothing scales and couples the predictor and corrector stages in a unified learning framework. We also include mathematical analysis, under well-posed system assumption, the limit of $\vec{p}_{\tau}$ converges to the optimal parameter in Appendix~\ref{convergence}

\subsection{Hyperparameter Robustness of DePCoN}
As shown in \cref{fig:1}-(d), HADES-NN exhibits pronounced sensitivity to the smoothing hyperparameter. Motivated by this limitation, we investigate how the performance of the proposed DePCoN framework depends on its smoothing hyperparameter choices. To this end, we examine the parameter estimation behavior of DePCoN across different values of the smoothing-grid size $N=4, 8$, and $12$ (\cref{fig:3}-(a)). Unlike HADES-NN, whose estimation accuracy varies significantly with the inner-loop iteration number $M$, the estimates produced by DePCoN remain consistently concentrated as $N$ varies, indicating substantially reduced hyperparameter sensitivity. See also \cref{fig:6}-(a), which summarizes DePCoN’s sensitivity to the choice of $N$ using box plots of the mean absolute percentage error (MAPE; \cref{subsec:metric}) over the range $N=4,\ldots,12$.

We next evaluate the computational efficiency of HADES-NN and DePCoN under different number of inner-loop iteration (HADES-NN) and smoothing scales (DePCoN). To this end, we plot MAPE profiles over a 1-hour time interval on the circadian pacemaker model (\cref{fig:3}-(b)), showing that DePCoN rapidly enters a low-error regime, reaching a MAPE of approximately $0.01$ within the first $0.03$h, and maintains stable convergence thereafter.
In contrast, HADES-NN converges more slowly and saturates at substantially higher error levels (MAPE=0.3), requiring longer optimization to reach comparable accuracy.  Together, these results demonstrate that DePCoN not only reduces sensitivity in the final parameter estimates but also accelerates convergence while stabilizing the overall optimization trajectory under hyperparameter variation.

Table~\ref{tab:my-table} summarizes the parameter estimation results for the circadian pacemaker model.
For each method, we report the mean $\pm$ standard deviation of the estimated parameter values across independent trials, together with the overall MAPE.
Consistent with the dispersion patterns observed in Figures~1 and~5, DePCoN yields estimates that are both accurate and stable, whereas the other methods exhibit large variability or systematic deviation from the ground truth.

\subsection{Robust and Consistent Parameter Estimation on a Lotka--Volterra Benchmark}
We next evaluate DePCoN on a non-autonomous Lotka--Volterra (LV) system, describing the interaction between prey ($y_1$) and predator ($y_2$) under a time-varying environment driven by $S(t)$, such as climate change, environmental pollution, and human intervention (\cref{fig:4}-(a)). Such environment changes can vary abruptly, and also yields non-smooth trajectories $\vec{y}=(y_1,y_2)$ (\cref{fig:4}-(a) right). This LV system serves as a benchmark example in which discontinuous external inputs are known to severely degrade the performance of gradient-based parameter estimation.
Although the LV model itself is relatively simple, it has been widely used to demonstrate the instability of optimization under discontinuous forcing, including in the benchmark study of ~\cite{doi:10.1137/25M1741340}.
This makes it a suitable testbed for assessing robustness rather than expressive power.

While the original benchmark considered $S(t)$ applied only to $y_1$, introducing opposing effects on both states further exacerbates non-smoothness in the loss landscape.
Under this modified setting, We apply a range of widely used optimization methods (DE, L-BFGS, LM, NM, SLSQP, and HADES-NN). Notably, HADES-NN is evaluated using hyperparameters reported to perform well in prior studies. Even under this favorable condition, none of the methods reliably recover the ground-truth parameters \cref{fig:4}-(b--d).
We then compare DePCoN with these six algorithms under identical experimental conditions. Each method is evaluated over 30 independent trials with parameters initialized uniformly from $[p_{\mathrm{o},i}/4,\,4p_{\mathrm{o},i}]$ for $i=1,\ldots,4$.

The results show that DePCoN is the only method that consistently yields estimates tightly clustered around the ground-truth parameters, as indicated by the concentration near the dashed-line intersection (\cref{fig:4}-(d)).
This behavior reflects both high accuracy and strong consistency across random initializations.
In contrast, competing methods exhibit either broad dispersion or systematic bias away from the ground truth, highlighting their sensitivity to both initial conditions and discontinuous inputs (\cref{fig:4}-(b--c)). 
In addition to the scatter plots in \cref{fig:4}-(d), \cref{tab:my-table} (Lotka Volterra) reports per-parameter estimation results (mean $\pm$ standard deviation) and overall MAPE, showing that DePCoN consistently outperforms all baselines.

We further examine the sensitivity of DePCoN to the choice of the smoothing-grid size $N$ using box plots of MAPE (See \cref{fig:6} (b) for details).
Across a wide range of values ($N=4,\ldots,12$), the MAPE distributions remain remarkably stable, with no significant performance degradation as $N$ varies.
Although increasing $N$ yields a denser set of smoothing scales and slightly higher computational cost, estimation accuracy remains largely invariant.
These results demonstrate that DePCoN achieves stable and reliable parameter estimation on a benchmark system where discontinuous inputs render conventional and smoothing-based methods highly sensitive to hyperparameter choices.

\section{Methodology}
\subsection{Description of the general non-autonomous differential equations with time-varying exogenous input $S(t)$} A system under the exogenous input, $S:[0,T]\rightarrow\mathbb{R}$, can generally be described by a non-autonomous differential equation with  $F:\mathbb{R}^{d_y}\times\mathbb{R}^{d_p}\times\mathbb{R}\rightarrow\mathbb{R}^{d_y}$:
\begin{equation}\label{de}
    \frac{d}{dt}\vec{y}=F(\vec{y},\vec{p},S)
\end{equation}
where $\vec{y}(t;\vec{p})=\vec{y}(t)\in\mathbb{R}^{d_y}$ denote system state of our interest at time $t$, $\vec{p}\in\mathbb{R}^{d_p}$ is unknown parameter vector that determine the shape of $\vec{y}$. Fitting the solution of system~\cref{de} to real observations requires the optimal choice of $\vec{p}$ so that the resulting trajectory matches the observation data $\{\vec{y}_{\mathrm{o}}(t_i):i=1,\ldots,N_{\mathrm{o}}\}$ on a finite time points $0=t_1<\ldots<t_{N_{\mathrm{o}}} =T$. This task is commonly formulated as a parameter estimation (inverse) problem: find an optimal $\vec{p}$ that minimizes the following loss function measuring the discrepancy between the solution $\vec{y}$ of~\cref{de} and $\vec{y}_{\mathrm{o}}$ at observation times $\{t_{i}\}_{i=1}^{N_{\mathrm{o}}}$:
\begin{align}
L(\vec{p}) & =\frac{1}{N_{\mathrm{o}}}\sum_{i=1}^{N_{\mathrm{o}}}\|\vec{y}(t_i;\vec{p})-\vec{y}_{\mathrm{o}}(t_i)\|_{2}^{2},\label{L}\\
& = \frac{1}{N_{\mathrm{o}}}\sum_{i=1}^{N_{\mathrm{o}}}\sum_{j=1}^{d_y}\left|y_j(t_i;\vec{p})-y_{o,j}(t_i)\right|^2.\nonumber
\end{align}
Under this setting, we denote $\vec{p}_{\mathrm{o}}$ by the minimizer of~\cref{L}. As noted in the introduction, this optimization problem is challenging in practice because the discontinuities in $S$ can make the map $\vec{p}\mapsto L(\vec{p})$ highly irregular. 

\begin{table*}[t]
\caption{Quantitative comparison of parameter estimation results.
Mean $\pm$ standard deviation of estimated parameter values and overall MAPE
across independent trials for DePCoN, HADES-NN~\cite{doi:10.1137/25M1741340},
and widely used optimization methods
(DE~\cite{price2006differential},
L-BFGS~\cite{zhu1997algorithm},
LM~\cite{more2006levenberg},
NM~\cite{gao2012implementing},
SLSQP~\cite{kraft1988software}).
}
\label{tab:my-table}
\centering
\begin{threeparttable}
\small
\setlength{\tabcolsep}{4pt}
\renewcommand{\arraystretch}{1.15}

\begin{tabular}{lccccccccc}
\toprule
Systems & Param. & True
& DePCoN (ours) & HADES-NN & DE & L-BFGS & LM & NM & SLSQP \\
\midrule
\multirow{5}{*}{\begin{tabular}[c]{@{}l@{}}Circadian\\pacemaker\end{tabular}}
& $\tau_c$ & 24   & \textbf{23.65$\pm$0.14} & 23.10$\pm$2.53 & $-\pm-$ & 51.08$\pm$27.43 & 51.08$\pm$27.43 & 51.08$\pm$27.43 & 52.18$\pm$31.42 \\
& $G$      & 20   & \textbf{19.85$\pm$0.30} & 19.59$\pm$1.21 & $-\pm-$ & 46.34$\pm$20.58 & 46.34$\pm$20.58 & 46.34$\pm$20.58 & 49.34$\pm$24.07 \\
& $\gamma$ & 0.23 & \textbf{0.23$\pm$0.00} & 0.28$\pm$0.08 & $-\pm-$ & 0.52$\pm$0.26 & 0.52$\pm$0.26 & 0.52$\pm$0.26 & 0.50$\pm$0.30 \\
& $k$      & 0.55 & \textbf{0.53$\pm$0.01} & 0.72$\pm$0.31 & $-\pm-$ & 1.09$\pm$0.64 & 1.09$\pm$0.64 & 1.09$\pm$0.64 & 1.06$\pm$0.82 \\
& \multicolumn{2}{c}{MAPE} & \textbf{0.02$\pm$0.01} & 0.16$\pm$0.24 & $-\pm-$ & 1.34$\pm$0.48 & 1.34$\pm$0.48 & 1.34$\pm$0.48 & 1.46$\pm$0.43 \\
\midrule
\multirow{5}{*}{\begin{tabular}[c]{@{}l@{}}Lotka\\Volterra\end{tabular}}
& $p_1$ & 2   & \textbf{1.97$\pm$0.03} & 3.00$\pm$1.92 & 5.97$\pm$2.00 & 4.02$\pm$2.70 & 4.17$\pm$2.16 & 5.51$\pm$3.06 & 5.04$\pm$2.50 \\
& $p_2$ & 0.5 & \textbf{0.49$\pm$0.01} & 0.56$\pm$0.38 & 0.83$\pm$0.52 & 1.01$\pm$0.61 & 0.60$\pm$0.42 & 0.85$\pm$0.33 & 1.27$\pm$0.61 \\
& $p_3$ & 1   & \textbf{0.99$\pm$0.01} & 1.47$\pm$0.68 & 2.45$\pm$0.74 & 2.27$\pm$1.02 & 2.03$\pm$0.74 & 2.55$\pm$0.83 & 2.25$\pm$1.14 \\
& $p_4$ & 1   & \textbf{0.99$\pm$0.01} & 1.19$\pm$0.57 & 1.71$\pm$0.95 & 1.93$\pm$1.05 & 0.91$\pm$0.49 & 1.26$\pm$0.58 & 2.81$\pm$1.07 \\
& \multicolumn{2}{c}{MAPE} & \textbf{0.02$\pm$0.01} & 0.53$\pm$0.56 & 1.37$\pm$0.62 & 1.29$\pm$0.58 & 0.84$\pm$0.51 & 1.22$\pm$0.53 & 1.65$\pm$0.62 \\
\bottomrule
\end{tabular}
\end{threeparttable}
\end{table*}

\subsection{Generation of a family of smoothed input via Heat-kernel convolution}\label{subsec:smoothing} We propose the Heat-kernel-based smoothing method. Specifically, let $K_{\tau}(t)=\frac{1}{\sqrt{4\pi\tau}}\exp\!\Big(-\frac{t^2}{4\tau}\Big)$ denote the heat kernel with a diffusion time $\tau>0$, (i.e., a smoothing scale: larger $\tau$ yields stronger smoothing) \cite{evans2022partial}. To define the convolution on $\mathbb{R}$, we extend the input $S:[0,T]\to\mathbb{R}$ by zero outside $[0,T]$. We then define $S_{\tau}$ as the convolution of $S$ with $K_{\tau}$:
\begin{equation}\label{smoothing_input}
S_{\tau}(t)
= (K_{\tau} * S)(t)
= \int_{-\infty}^{\infty} K_{\tau}(t-s)\,S(s)\,ds,\text{ } \forall\tau>0.
\end{equation}
By construction, the heat-kernel convolution yields a smooth input $S_{\tau}\in C^{\infty}(\mathbb{R})$. Moreover, as $\tau\to0^+$, $S_{\tau}\to S$ in the appropriate function-space sense (see Appendix~\ref{convergence} and \cite{stein2009real}). Finally, by selecting $N+1$ discrete values of $\tau$ on $[0,1]$,
$\{\tau_{n}=n/N\}_{n=0}^{N}$, we obtain a family of smoothed inputs
$\{S_{\tau_n}\}_{n=0}^{N}$ with $S_{\tau_0}=S$ (\cref{fig:2}-(a)).

\subsection{Development of the predictor network}\label{subsec:pred_net}
We introduce a predictor network $f_{\theta}$ with trainable parameters $\theta$ (\cref{fig:2}-(b)), implemented as a fully connected neural network. Given the parameter estimate $\vec{p}_{\tau_n}$, the predictor produces the next estimate at the smoothing scale $\tau_{n-1}$, $\vec{p}_{\tau_{n-1}} = f_{\theta}(\vec{p}_{\tau_n})$, $\forall n=1,\ldots,N$. To this end, at the coarsest scale $\tau_N$, we initialize the parameter $\vec{p}_{\tau_N}$ by sampling it from the admissible set $P_{\mathrm{o}}$ (see Assumption~\ref{assumptionB1} for details). Starting from the coarsest level $\tau_N$, we feed $\vec{p}_{\tau_N}$ into $f_{\theta}$. The first layer maps $\vec{p}_{\tau_N}$ to a 64 nodes via a linear transformation followed by a ReLU activation. This operation is then repeated through two additional hidden layers with 32 and 16 units, respectively, each followed by a ReLU activation, yielding a final 16-dimensional feature vector. The output layer is a single fully connected linear layer (no activation) that maps the final 16 features to the next parameter estimate $\vec{p}_{\tau_{N-1}}$ at smoothing scale $\tau_{N-1}$. Repeating this procedure recursively generates $\{\vec{p}_{\tau_n}\}_{n=0}^{N-1}$ (\cref{fig:2}-(b) Parameters).

\subsection{Development of the corrector network}\label{subsec:corr_net}
Although the predictor yields a sequence of parameter estimates $\{\vec{p}_{\tau_n}\}_{n=0}^{N}$, these predictions are not, in general, optimal for any well-defined objective at each smoothing level yet. We therefore introduce a correction step. For each smoothing scale $\tau_n$, given the predicted parameter $\vec{p}_{\tau_n}$ from the previous step, we first compute the state trajectory $\vec{y}_{\tau_n}$ by solving the $\tau$-dependent non-autonomous system (\cref{fig:2}-(b)):
\begin{equation}\label{de_tau}
        \frac{d}{dt}\vec{y}=F(\vec{y},\vec{p},S_{\tau_n}).
\end{equation}
In practice, the solution operator for~\cref{de_tau} is implemented using a Neural ODE solver~\cite{chen2018neural}, which enables end-to-end differentiable evaluation of $\vec{y}_{\tau_n}$ and supports gradient-based updates of both $\vec{p}_{\tau_n}$ and the predictor parameters $\theta$ via adjoint-based backpropagation (\cref{fig:2}-(c), NeuralODE).

We then quantify, for each smoothing scale $\tau$, the discrepancy between $\vec{y}_{\tau}$ and $\vec{y}_{\mathrm{o}}$ by evaluating the following $\tau$-dependent loss at observation times:
\begin{equation}\label{L_tau}
L_{\tau}(\vec{p})=\frac{1}{N_{\mathrm{o}}}\sum_{i=1}^{N_{\mathrm{o}}}\|\vec{y}_{\tau}(t_i;\vec{p})-\vec{y}_{\mathrm{o}}(t_i)\|_2^2,
\end{equation}
for $\tau=\tau_0,\ldots,\tau_N$ (\cref{fig:2}-(c) Optimization). We then aggregate the losses across smoothing levels to obtain the multi-scale objective
\begin{equation}\label{L_total}
\sum_{n=0}^{N} L_{\tau_n}(\vec{p})
=\frac{1}{N_{\mathrm{o}}}\sum_{n=0}^{N}\sum_{i=1}^{N_{\mathrm{o}}}\big\|\vec{y}_{\tau_n}(t_i;\vec{p})-\vec{y}_{\mathrm{o}}(t_i)\big\|_2^2.
\end{equation}
Finally, we update both the predictor parameters $\theta$ and the terminal parameter $\vec{p}_{\tau_N}$ using the Adam optimizer. We repeat this predictor--corrector procedure until the gradient $\nabla_{\theta}\sum_{n=0}^{N} L_{\tau_n}(\vec{p}_{\tau_n})$ becomes sufficiently small.
\section{Conclusion}
Our results demonstrate that DePCoN consistently achieves stable and accurate parameter estimation under discontinuous inputs, outperforming widely used optimization methods, together with neural smoothing approaches in terms of robustness and time-to-accuracy.
Beyond the specific setting of non-autonomous dynamical systems, DePCoN offers a general learning principle for stabilizing optimization under non-smooth objectives.
By integrating heat-kernel convolution into the learning dynamics and enforcing consistency across scales, our framework transforms a traditionally fragile hyperparameter choice into a structured learning input.
This perspective is broadly applicable to machine learning problems involving discontinuous conditioning, hybrid dynamics, or highly non-smooth loss landscapes.

\newpage
\bibliography{example_paper}
\bibliographystyle{icml2026}

\section*{Acknowledgements}
We do not include acknowledgements in the initial version of the paper
submitted for blind review.

\section*{Impact Statement}
This paper presents work whose goal is to advance the field of Machine Learning for scientific applications. There are many potential societal consequences of our work, none of which we feel must be specifically highlighted here.

\newpage
\appendix
\onecolumn
\section{Preliminaries}
\subsection{Notation}
Let $\vec{p}=(p_1,p_2,\ldots,p_{d_p})\in\mathbb{R}^{d_p}$ denote a real-valued vector of dimension $d_p\ge 1$. We measure its magnitude using the $L^2$ (Euclidean) norm:
$$\|\vec{p}\|_2=\left(\sum_{i=1}^{d_p}p_i^2\right)^{1/2}.$$

For a time-dependent vector $\vec{y}(t)=(y_1(t),y_2(t),\ldots,y_{d_y}(t))\in\mathbb{R}^{d_y}$ defined on $t\in[0,T]$, we use the pointwise $L^2$ norm
$$\|\vec{y}(t)\|_2=\left(\sum_{i=1}^{d_y}y_i^2(t)\right)^{1/2}.$$

\subsection{Function spaces}
$\mathcal{L}^2(X)$ denotes the set of real-valued Lebesgue measurable functions defined on $X\subset\mathbb{R}^{d_x}$ $f:X\to\mathbb{R}$ such that $\|f\|_{\mathcal{L}^2(X)}:=\Big(\int_{X} |f(t)|^2\,dt\Big)^{1/2}<\infty.$ 

For vector-valued functions $\vec{y}:X\rightarrow\mathbb{R}^{d_y}$, same notations are used:
$\|\vec{y}\|_{\mathcal{L}^2(X)}:=\Big(\sum_{i=1}^{d_y}\int_{X} y_i^2(t)\,dt\Big)^{1/2}<\infty.$

\subsection{Mean absolute percentage error}\label{subsec:metric}
We evaluate parameter-estimation accuracy using the mean absolute percentage error (MAPE) and its averaged variants.
Let $\vec{p}_{\mathrm{o}}=(p_{\mathrm{o},1},\ldots,p_{\mathrm{o},d_p})\in\mathbb{R}^{d_p}$ denote the ground-truth parameter vector and $\vec{p}_{j}=(p_{j,1},\ldots,p_{j,d_p})$ the estimate obtained from the $j$-th independent trial ($j=1,\dots,K$).
We define the elementwise absolute percentage error (APE) as
\begin{equation}\nonumber
\mathrm{APE}_{j,i}=\frac{|p_{j,i}-p_{\mathrm{o},i}|}{|p_{\mathrm{o},i}|},\qquad i=1,\dots,d_p.
\end{equation}
The MAPE is then defined by averaging over parameters and trials:
\begin{equation}\nonumber
\mathrm{MAPE}=\frac{1}{K}\sum_{j=1}^{K}\frac{1}{d_p}\sum_{i=1}^{d_p}\mathrm{APE}_{j,i}.
\end{equation}

\section{Convergence of DePCoN estimates $\vec{p}_{\tau}\to\vec{p}_{\mathrm{o}}$ as $\tau\to 0^+$}\label{convergence}

\*paragraph{Purpose of the analysis.}
The following convergence analysis shows that, under standard well-posedness and identifiability assumptions, the sequence of optimal parameter estimates obtained at each smoothing scale is well-defined and converges to the optimizer of the unsmoothed problem as the smoothing vanishes. While this result assumes access to the scale-wise optima and therefore does not model the exact training dynamics of DePCoN, it provides the theoretical rationale for smoothing-based continuation: in practice, DePCoN learns to track these scale-wise optima across multiple smoothing levels and, at the finest scale, produces the final estimate by inferring the $\tau=0$ parameter through the learned predictor network.

Let $(\vec{y}_{\mathrm{o}}(t), \vec{p}_{\mathrm{o}}, S(t))$ denote the observation trajectory, the ground-truth (optimal) parameter associated with~\cref{L}, and the given exogenous input, respectively. The underlying true system is
\begin{equation}\label{de_optimal}
    \frac{d}{dt}\vec{y}_{\mathrm{o}}(t)=F(\vec{y}_{\mathrm{o}}(t),\vec{p}_{\mathrm{o}},S(t)).
\end{equation}

Although observations are available only at finitely many time points, we treat them as a continuous-time trajectory for notational simplicity and theoretical clarity; the same argument applies to the discretized setting. 

For each $\tau>0$, let $\vec{y}_{\tau}(t)=\vec{y}_{\tau}(t;\vec{p})$ denote the solution of the $\tau$-dependent system~\cref{de_tau},
$$\frac{d}{dt}\vec{y}_{\tau}(t)=F(\vec{y}_{\tau}(t),\vec{p},S_{\tau}(t)),$$
and let $\vec{y}(t)=\vec{y}(t;\vec{p})$ denote the solution of~\cref{de}. By~\cref{de_optimal}, we have $\vec{y}(t;\vec{p}_{\mathrm{o}})=\vec{y}_{\mathrm{o}}(t)$.

In practice, the parameter estimates $\vec{p}_{\tau}$ are obtained by training DePCoN in the correction stage by minimizing~\cref{L_tau}. In this section, we show that $\vec{p}_{\tau}$ converges to the ground-truth parameter $\vec{p}_{\mathrm{o}}$ in the vanishing-smoothing limit $\tau\to 0^+$. This convergence analysis also motivates the prediction stage of DePCoN, which propagates parameter estimates across smoothing scales. To establish the result, we impose the following assumptions, which ensure stability of the dynamical system.

\begin{assumption}\label{assumptionB1}[Admissible set and nondegenerate stationary branch]
For each $\tau>0$, let $\vec p_\tau$ denote a selected local minimizer of the (continuous-time) $\tau$-dependent loss $L_{\tau}$ in~\cref{L_tau}:
\begin{equation}\label{L_tau_cont}
    L_{\tau}(\vec{p})=\|\vec{y}_{\tau}(\cdot;\vec{p})-\vec{y}_{\mathrm{o}}(\cdot)\|_{\mathcal{L}^2([0,T])}^2.
\end{equation}
We assume that there exists a nonempty compact set $P_{\mathrm{o}}\subset\mathbb{R}^{d_p}$ such that
$$\{\vec p_\tau\}_{\tau>0}\subset P_{\mathrm{o}}.$$
Moreover, there exists  $\tau_*>0$ such that the selected branch $\{\vec p_\tau\}_{0<\tau\leq\tau_*}$ consists of nondegenerate stationary points. That is, for all $0<\tau\leq\tau_*$,
$$\nabla_{\vec p} L_\tau(\vec p_\tau)=\vec 0,\text{ }
\nabla^2_{\vec p} L_\tau(\vec p_\tau)\ \text{is invertible}.$$
\end{assumption}

We adopt Assumption~\ref{assumptionB1} because the parameters $\vec{p}_{\tau}$ typically represent domain-specific quantities (e.g., physical or biological constants) whose plausible ranges are often known a priori. In our examples, we set
$$P_{\mathrm{o}}=\Big[\tfrac{1}{4}p_{\mathrm{o},1},\,4p_{\mathrm{o},1}\Big]\times\cdots\times
\Big[\tfrac{1}{4}p_{\mathrm{o},d_p},\,4p_{\mathrm{o},d_p}\Big].$$
Accordingly, such bounds can be imposed directly in the optimization procedure, preventing $\vec{p}_{\tau}$ from leaving the admissible set $P_{\mathrm{o}}$ and excluding nonphysical regimes. 

The nondegeneracy condition on the stationary branch $\{\vec{p}_{\tau}\}$ is imposed to ensure local identifiability and stability of the minimizers with respect to perturbations in $\tau$. In particular, invertibility of the Hessian $\nabla_{\vec{p}}^{2}L_{\tau}(\vec{p}_{\tau})$ rules out flat directions and guarantees that the stationary points are isolated.

\begin{assumption}\label{assumptionB2}[Well-posedness and regularity of $F$ in~\cref{de}]
There exist constants $C_1,C_2>0$ such that, for all $\vec p\in P_{\mathrm{o}}$, $\vec{y}, \vec{y}_{1},\vec{y}_{2}\in\mathbb{R}^{d_y}$, and $S,S_1,S_2\in\mathbb{R}$,
\begin{align}
\|F(\vec{y}_{1},\vec p,S)-F(\vec{y}_{2},\vec p,S)\|_2
&\le C_1\|\vec{y}_{1}-\vec{y}_{2}\|_2,\nonumber\\
\|F(\vec{y},\vec p,S_1)-F(\vec{y},\vec p,S_2)\|_2
&\le C_2|S_1-S_2|.\nonumber
\end{align}
\end{assumption}

Under Assumption~\ref{assumptionB2}, for each $\vec p\in P_{\mathrm{o}}$ and each initial condition, the initial value problem~\cref{de} admits a unique solution $\vec{y}(\cdot;\vec{p})$ on $[0,T]$. Consequently, the loss $L(\vec{p})$ is well-defined.

In addition, $F$ is continuously differentiable with respect to $(\vec{y},\vec{p})$ on $\mathbb{R}^{d_y}\times P_{\mathrm{o}}\times\mathbb{R}$, and its partial derivatives $\partial_{\vec{y}}F$ and $\partial_{\vec{p}}F$ are locally bounded on this set.

\begin{assumption}\label{assumptionB3}[Optional: identifiability and local stability of the parameter]
Given the input $S\in\mathcal{L}^2([0,T])$ and the observation trajectory $\vec{y}_{\mathrm{o}}(t)$, consider the continuous-time loss $L(\vec{p})$ in~\cref{L}:
\begin{equation}\label{L_cont}
    L(\vec{p})=\|\vec{y}(\cdot;\vec{p})-\vec{y}_{\mathrm{o}}(\cdot)\|_{\mathcal{L}^2([0,T])}^2,
\end{equation}
where $\vec{y}(\cdot;\vec{p})$ denotes the solution of~\cref{de} under parameter $\vec{p}$. In general, the minimizer of~\cref{L_cont} need not be unique, so $\arg\min_{\vec{p}\in P_{\mathrm{o}}} L(\vec{p})$ is naturally viewed as a set of minimizers.
We assume that $L$ admits a unique global minimizer $\vec{p}_{\mathrm{o}}\in P_{\mathrm{o}}$, i.e., $\arg\min_{\vec{p}\in P_{\mathrm{o}}} L(\vec{p})={\vec{p}_{\mathrm{o}}}$, and that $L$ satisfies a quadratic growth condition at $\vec{p}_{\mathrm{o}}$: there exist a constants $c>0$ such that
$$L^{1/2}(\vec{p})\ \ge\ L^{1/2}(\vec{p}_{\mathrm{o}})\ +\ c\,\|\vec{p}-\vec{p}_{\mathrm{o}}\|_2^2,
\qquad \forall\,\vec{p}\in P_{\mathrm{o}}.$$
\end{assumption}
Assumption~\ref{assumptionB3} is optional but conceptually decisive. It encodes identifiability and well-conditioning: the uniqueness of $\vec{p}_{\mathrm{o}}$ rules out structural non-identifiability (multiple parameters producing indistinguishable trajectories), while the quadratic growth condition excludes flat directions near the optimum. As a result, convergence statements for $\vec{p}_{\tau}$ bifurcate depending on whether Assumption~\ref{assumptionB3} is imposed. Without Assumption~\ref{assumptionB3}, one can typically guarantee only that $\{\vec{p}_{\tau}\}_{\tau>0}$ has accumulation points in $P_{\mathrm{o}}$ and that every limit point belongs to the set of global minimizers $\arg\min_{\vec{p}\in P_{\mathrm{o}}}L(\vec{p})$; in contrast, under Assumption~\ref{assumptionB3}, the limit is forced to be the single identifiable parameter $\vec{p}_{\mathrm{o}}$.

Under Assumptions~\ref{assumptionB1}--\ref{assumptionB3}, we show that, as $\tau\to 0^+$, the minimizers $\{\vec{p}_{\tau}\}_{\tau>0}\subset P_{\mathrm{o}}$ of the $\tau$-dependent loss $L_{\tau}$ in~\cref{L_tau} converge to the ground-truth parameter $\vec{p}_{\mathrm{o}}$. The proof is organized into four steps:

\textbf{Step B.1 (Input mollification)}
As $\tau\to0^+$, $\|S_{\tau}-S\|_{\mathcal{L}^2(\mathbb{R})}\to0,$ by Lemma~\ref{lem:mollify_Lp} and Corollary~\ref{cor:heat_moli}. Obviously, this implies $\|S_{\tau}-S\|_{\mathcal{L}^2([0,T])}\to0$.

\textbf{Step B.2 (Stability of trajectories)}
Fix $\vec{p}\in P_{\mathrm{o}}$, and let $\vec{y}_{\tau}(\cdot;\vec{p})$ and $\vec{y}(\cdot;\vec{p})$ be the solutions of~\cref{de_tau} and~\cref{de}, respectively. Under Assumption~\ref{assumptionB2},
$$\|S_{\tau}-S\|_{\mathcal{L}^2([0,T])}\to0
\quad\text{implies}\quad\|\vec{y}_{\tau}(\cdot;\vec{p})-\vec{y}(\cdot;\vec{p})\|_{\mathcal{L}^2([0,T])}\to0,$$
by Step B.1 and Lemma~\ref{lem:lipschitz}.

\textbf{Step B.3 (Stability of the loss)}
If $\tau\to0^+$, then
$$\|L_{\tau}^{1/2}-L^{1/2}\|_{\mathcal{L}^2(P_{\mathrm{o}})}\to0,$$
by Step B.2 and Corollary~\ref{cor:ytoL}, where $L_{\tau}(\vec{p})$ and $L(\vec{p})$ are defined in~\cref{L_tau_cont} and~\cref{L_cont}, respectively.

\textbf{Step B.4 (Convergence of minimizers).}
\emph{(i) Without Assumption~\ref{assumptionB3}.} The stability of the objectives implies convergence in optimal value:
$$\lim_{\tau\to0^+} L(\vec{p}_{\tau})=L(\vec{p}_{\mathrm{o}}),$$
and hence every accumulation point of $\{\vec{p}_{\tau}\}_{\tau>0}$ (which exists up to subsequence by compactness of $P_{\mathrm{o}}$ in Assumption~\ref{assumptionB1}) is a global minimizer of $L$; equivalently, if $\lim_{\tau\to0^+}\vec{p}_{\tau}$ exists, then the limit belongs to minimizers of $L(\vec{p})$.

\emph{(ii) Under Assumption~\ref{assumptionB3}.} The minimizers themselves converge: $\|\vec{p}_{\tau}-\vec{p}_{\mathrm{o}}\|_2\to0$ as $\tau\to0^+.$
Moreover, combining Corollary~\ref{cor:delta_tau} with Theorem~\ref{thm:p_convergence} yields the quantitative bound
$$\|\vec{p}_{\tau}-\vec{p}_{\mathrm{o}}\|_2 = \mathcal{O}(\delta_{\tau}^{1/2})\qquad(\tau\to0^+),$$
where $\delta_{\tau}:=\sup_{\vec{p}\in P_{\mathrm{o}}}\big|L_{\tau}^{1/2}(\vec{p})-L^{1/2}(\vec{p})\big|$.

\*paragraph{Remark}
DePCoN does not explicitly compute $\arg\min L_\tau$ for each $\tau$.
Rather, it minimizes an aggregated multi-scale loss that couples parameter estimates across smoothing levels through a learned predictor.
The consistency result above suggests that, when the multi-scale loss is well-optimized, the learned parameters are encouraged to remain close to the family of minimizers of $L_\tau$, thereby stabilizing convergence toward the $\tau \to 0$ regime.

\begin{lemma}\label{lem:mollify_Lp}
Let $S\in\mathcal{L}^2(\mathbb{R})$ be an input supported on $[0,T]$, extended by $0$ to $\mathbb{R}\setminus[0,T]$.
Assume that $\{K_\tau\}_{\tau>0}\subset C^\infty(\mathbb{R})$ is a family of nonnegative kernels satisfying
\begin{equation}\label{mollify_condition}
    \int_{\mathbb{R}} K_\tau(t)\,dt = 1 \quad \forall\tau>0,
 \text{ and } 
\lim_{\tau\rightarrow0^+}\int_{|t|>\delta} K_\tau(t)\,dt=0 \ \forall\delta>0.
\end{equation}
Then $\|K_\tau * S - S\|_{\mathcal{L}^2(\mathbb{R})}\to 0
\text{ as }\tau\to0^+,$ where, for each $\tau>0$,
$$(K_{\tau} * S)(t)=\int_{-\infty}^{\infty} K_{\tau}(t-s)\,S(s)\,ds,$$ and $K_\tau * S\in C^{\infty}(\mathbb{R})$.
\end{lemma}
\begin{proof}
This is a standard approximation-to-the-identity result in $\mathcal{L}^2(\mathbb{R})$; see, \cite{stein2009real,brezis2011functional}.
\end{proof}

\begin{corollary}\label{cor:heat_moli}
Let
$$K_{\tau}(t)=\frac{1}{\sqrt{4\pi\tau}}\exp\!\Big(-\frac{t^2}{4\tau}\Big)$$
be the heat kernel and define $S_{\tau}:=K_{\tau}*S$. Suppose $S\in \mathcal{L}^2(\mathbb{R})$ and $S$ is supported on $[0,T]$. Then:
\begin{align}
\text{i) } & \|S_{\tau}-S\|_{\mathcal{L}^2(\mathbb{R})}\to0\text{ as }tau\to0^+,\label{approx_id_S}\\
\text{ii } & \|S_{\tau}\|_{\mathcal{L}^2(\mathbb{R})}\to0\text{ as } \tau\to\infty.\label{diffusion_S}
\end{align}
\end{corollary}
\begin{proof}
i) The heat kernel $K_{\tau}$ satisfies~\cref{mollify_condition}. Therefore, by Lemma~\ref{lem:mollify_Lp},
$$\|S_{\tau}-S\|_{\mathcal{L}^2(\mathbb{R})}\to0
\quad\text{as}\quad \tau\to0^+.$$
ii) Since $S$ is supported on $[0,T]$, we have $S\in\mathcal{L}^1(\mathbb{R})$ ($\mathcal{L}^{p}$ embedding). By Young's inequality,
$$\|S_{\tau}\|_{\mathcal{L}^2(\mathbb{R})}
=\|K_{\tau}*S\|_{\mathcal{L}^2(\mathbb{R})}
\le \|K_{\tau}\|_{\mathcal{L}^2(\mathbb{R})}\,\|S\|_{\mathcal{L}^1(\mathbb{R})}.$$
For the heat kernel, $\|K_{\tau}\|_{\mathcal{L}^2(\mathbb{R})}\propto \tau^{-1/4}$, and hence $\|S_{\tau}\|_{\mathcal{L}^2(\mathbb{R})}\to0$ as $\tau\to\infty$.
\end{proof}

\begin{lemma}\label{lem:lipschitz} Fix $\vec p \in P_{\mathrm{o}}$ and $\tau_1,\tau_2>0$. Let $\vec y_{\tau_1}(t)=\vec y_{\tau_1}(t;\vec{p})$ and $\vec y_{\tau_2}(t)=\vec y_{\tau_2}(t;\vec{p})$ denote the solutions of~\cref{de_tau} driven by the smoothed inputs $S_{\tau_1}$ and $S_{\tau_2}$, respectively, with the same initial condition as the observation. That is,
\begin{align}
\frac{d}{dt}\vec y_{\tau_1}(t) &= F\!\big(\vec y_{\tau_1}(t),\vec p,S_{\tau_1}(t)\big), \label{de_ytau1}\\
\frac{d}{dt}\vec y_{\tau_2}(t) &= F\!\big(\vec y_{\tau_2}(t),\vec p,S_{\tau_2}(t)\big), \label{de_ytau2}\\
\vec y_{\tau_1}(0) &= \vec y_{\tau_2}(0)=\vec y_{\mathrm{o}}(0). \label{de_ytau_ic}
\end{align}
Then, under Assumption~\ref{assumptionB2}, the following stability estimate holds:
\begin{equation}\label{lipschitz}
\|\vec{y}_{\tau_1}(\cdot;\vec p)-\vec{y}_{\tau_2}(\cdot;\vec p)\|_{\mathcal{L}^{2}([0,T])}
\le C(T)\, \|S_{\tau_1}-S_{\tau_2}\|_{\mathcal{L}^{2}([0,T])},
\end{equation}
where the constant $C(T)>0$ depends only on $C_1$, $C_2$, and the terminal time $T\in(0,\infty)$.
\end{lemma}
\begin{proof}
Subtracting~\cref{de_ytau2} from~\cref{de_ytau1} and integrating over $[0,t]$ yield
\begin{align}
\vec{y}_{\tau_1}(t)-\vec{y}_{\tau_2}(t)
&=\int_{0}^{t}\Big(F(\vec{y}_{\tau_1}(v),\vec{p},S_{\tau_1}(v)) - F(\vec{y}_{\tau_1}(v),\vec{p},S_{\tau_2}(v))\Big)\,dv \nonumber\\
&\quad+\int_{0}^{t}\Big(F(\vec{y}_{\tau_1}(v),\vec{p},S_{\tau_2}(v)) - F(\vec{y}_{\tau_2}(v),\vec{p},S_{\tau_2}(v))\Big)\,dv, \nonumber
\end{align}
where we used the initial condition~\cref{de_ytau_ic} to simplify the left-hand side.

Taking $L^2$ norms, squaring both sides, and applying $(a+b)^2\le 2(a^2+b^2)$, we obtain
\begin{align}
\frac{1}{2}\left\Vert \vec{y}_{\tau_1}(t)-\vec{y}_{\tau_2}(t)\right\Vert_2 ^2
&\leq \left\Vert \int_{0}^{t}\Big[F(\vec{y}_{\tau_1}(v),\vec{p},S_{\tau_1}(v)) - F(\vec{y}_{\tau_1}(v),\vec{p},S_{\tau_2}(v)) \Big]\,dv\right\Vert_2^2 \nonumber\\
&\quad + \left\Vert\int_{0}^{t}\Big[F(\vec{y}_{\tau_1}(v),\vec{p},S_{\tau_2}(v)) - F(\vec{y}_{\tau_2}(v),\vec{p},S_{\tau_2}(v))\Big]\,dv\right\Vert_2^2. \label{ineq_F}
\end{align}

By H\"older's inequality and Assumption~\ref{assumptionB2}, the first term on the right-hand side of~\cref{ineq_F} is bounded as
\begin{align}
\left\Vert \int_{0}^{t}\Big[F(\vec{y}_{\tau_1}(v),\vec{p},S_{\tau_1}(v)) - F(\vec{y}_{\tau_1}(v),\vec{p},S_{\tau_2}(v)) \Big]\,dv\right\Vert_2^2 & \le C_2^2\, t \int_{0}^{t}\left|S_{\tau_1}(v)-S_{\tau_2}(v)\right|^2 dv\nonumber\\
& \le C_2^2\, T \int_{0}^{t}\left|S_{\tau_1}(v)-S_{\tau_2}(v)\right|^2 dv.\label{ineq_S}
\end{align}
where $C_2$ is the global constant as in the Assumption~\ref{assumptionB2}. Similarly, the second term on the right-hand side of~\cref{ineq_F} admits the bound
\begin{equation}\label{ineq_y}
\left\Vert \int_{0}^{t}\Big[F(\vec{y}_{\tau_1}(v),\vec{p},S_{\tau_2}(v)) - F(\vec{y}_{\tau_2}(v),\vec{p},S_{\tau_2}(v)) \Big]\,dv\right\Vert_2 ^2 \leq C_1^2\, T \int_{0}^{t}\left\Vert \vec{y}_{\tau_1}(v)-\vec{y}_{\tau_2}(v)\right\Vert_2 ^2 dv.
\end{equation}
By combining~\cref{ineq_F},~\cref{ineq_S}, and~\cref{ineq_y}, we obtain
\begin{equation}\label{ineq_simple}
\left\Vert \vec{y}_{\tau_1}(t)-\vec{y}_{\tau_2}(t)\right\Vert_2 ^2
\leq \bar{C}(T)\left(\int_{0}^{t}\left|S_{\tau_1}(v)-S_{\tau_2}(v)\right|^2\,dv
+\int_{0}^{t}\left\Vert \vec{y}_{\tau_1}(v)-\vec{y}_{\tau_2}(v)\right\Vert_2 ^2\,dv \right),
\end{equation}
where one may take $\bar{C}(T)=2T\max\{C_2^2,\,C_1^2\}$.

Applying Gr\"onwall's inequality \cite{ye2007generalized} to~\cref{ineq_simple} on $[0,T]$ yields, for all $t\in[0,T]$,
$$\|\vec{y}_{\tau_1}(t)-\vec{y}_{\tau_2}(t)\|_2^2
\le \tilde{C}(T)\int_0^t |S_{\tau_1}(v)-S_{\tau_2}(v)|^2\,dv.$$
Integrating this estimate over $t\in[0,T]$ gives
$$\|\vec{y}_{\tau_1}(\cdot;\vec p)-\vec{y}_{\tau_2}(\cdot;\vec p)\|_{\mathcal{L}^{2}([0,T])}
\le C(T)\,\|S_{\tau_1}-S_{\tau_2}\|_{\mathcal{L}^{2}([0,T])},$$
which proves~\cref{lipschitz}.
\end{proof}

\begin{corollary}\label{cor:ytoL}
Let $L_{\tau}(\vec{p})$ and $L(\vec{p})$ be the losses defined in~\cref{L_tau_cont} and~\cref{L_cont}, respectively. Then
\begin{equation}\label{ineq_ytoL}
\|L_{\tau}^{1/2}-L^{1/2}\|_{\mathcal{L}^{2}(P_{\mathrm{o}})}
\leq C\|S_{\tau}-S\|_{\mathcal{L}^{2}([0,T])},
\end{equation}
where the constant $C>0$ depends only on the terminal time $T$ and the set $P_{\mathrm{o}}$.
\end{corollary}

\begin{proof}
We first bound $\big|L_{\tau}^{1/2}(\vec{p})-L^{1/2}(\vec{p})\big|$. By the reverse triangle inequality in $\mathcal{L}^2([0,T])$,
\begin{align}
\big|L_{\tau}^{1/2}(\vec{p})-L^{1/2}(\vec{p})\big|
&=\Big|\|\vec{y}_{\tau}(\cdot;\vec{p})-\vec{y}_{\mathrm{o}}(\cdot)\|_{\mathcal{L}^2([0,T])}
-\|\vec{y}(\cdot;\vec{p})-\vec{y}_{\mathrm{o}}(\cdot)\|_{\mathcal{L}^2([0,T])}\Big| \nonumber\\
&\le \|\vec{y}_{\tau}(\cdot;\vec{p})-\vec{y}(\cdot;\vec{p})\|_{\mathcal{L}^2([0,T])}
\le C(T)\,\|S_{\tau}-S\|_{\mathcal{L}^{2}([0,T])}, \label{L_diff_sq}
\end{align}
where the last inequality follows from Lemma~\ref{lem:lipschitz} (with $\tau_1=\tau$ and $\tau_2=0$, i.e., $S_{\tau_2}=S$).

Squaring both sides of~\cref{L_diff_sq} and integrating over $\vec{p}\in P_{\mathrm{o}}$ yield
$$\|L_{\tau}^{1/2}-L^{1/2}\|_{\mathcal{L}^{2}(P_{\mathrm{o}})}
\le C(T)|P_{\mathrm{o}}|^{1/2}\\|S_{\tau}-S\|_{\mathcal{L}^{2}([0,T])},$$
where $|P_{\mathrm{o}}|$ denotes the Lebesgue measure of $P_{\mathrm{o}}$. By Assumption~\ref{assumptionB1}, $P_{\mathrm{o}}$ is compact in $\mathbb{R}^{d_p}$ and hence has finite measure.
\end{proof}

\begin{corollary}\label{cor:delta_tau}
Given $S\in\mathcal{L}^{2}(\mathbb{R})$, let $\delta_{\tau}:=\sup_{\vec{p}\in P_{\mathrm{o}}}\big|L_{\tau}^{1/2}(\vec{p})-L^{1/2}(\vec{p})\big|.$
Then, for any $T\in(0,\infty)$, $\delta_{\tau}$ is uniformly bounded over $\tau>0$. That is, there exists a constant $M>0$ such that
$$0\le \delta_{\tau}\le M,\qquad \forall\,\tau>0.$$
\end{corollary}

\begin{proof}
Recall~\cref{L_diff_sq}:
$$\big|L_{\tau}^{1/2}(\vec{p})-L^{1/2}(\vec{p})\big|
\le C(T)\,\|S_{\tau}-S\|_{\mathcal{L}^{2}([0,T])}.$$
By~\cref{diffusion_S} in Corollary~\ref{cor:heat_moli}, $\|S_{\tau}\|_{\mathcal{L}^{2}(\mathbb{R})}\to0$ as $\tau\to\infty$, and hence $\sup_{\tau>0}\|S_{\tau}\|_{\mathcal{L}^{2}(\mathbb{R})}<\infty$. Let
$$M_1:=\sup_{\tau>0}\|S_{\tau}\|_{\mathcal{L}^{2}(\mathbb{R})}.$$
Then, for all $\tau>0$,
$$\|S_{\tau}-S\|_{\mathcal{L}^{2}([0,T])}
\le \|S_{\tau}\|_{\mathcal{L}^{2}([0,T])}+\|S\|_{\mathcal{L}^{2}([0,T])}
\le M_1+\|S\|_{\mathcal{L}^{2}([0,T])}.$$
Therefore,
$$\big|L_{\tau}^{1/2}(\vec{p})-L^{1/2}(\vec{p})\big|
\le C(T)\big(M_1+\|S\|_{\mathcal{L}^{2}([0,T])}\big)=:M,$$
and taking the supremum over $\vec{p}\in P_{\mathrm{o}}$ proves the claim.
\end{proof}

\begin{theorem}\label{thm:p_convergence}
Let $\vec{p}_{\tau}\in P_{\mathrm{o}}$ be a minimizer of $L_{\tau}$ for each $\tau>0$.  

\emph{(i) Without Assumption~\ref{assumptionB3}.} We have
$$\lim_{\tau\to0^+} L(\vec{p}_{\tau}) = L(\vec{p}_{\mathrm{o}}).$$
In particular, if the limit $\lim_{\tau\to0^+}\vec{p}_{\tau}$ exists, then it belongs to the set of global minimizers of $L$ on $P_{\mathrm{o}}$, i.e.,
$$\lim_{\tau\to0^+}\vec{p}_{\tau}\in \arg\min_{\vec{p}\in P_{\mathrm{o}}} L(\vec{p}).$$

\emph{(ii) Under Assumption~\ref{assumptionB3}.} We further have
$$\vec{p}_{\tau}\to \vec{p}_{\mathrm{o}}
\quad\text{as}\quad \tau\to0^+,$$
and there exists a constant $C>0$ such that
$$\|\vec{p}_{\tau}-\vec{p}_{\mathrm{o}}\|_2 \le C\,\delta_{\tau},
\qquad \tau\to0^+,$$
where $\delta_{\tau}=\sup_{\vec{p}\in P_{\mathrm{o}}}\big|L_{\tau}^{1/2}(\vec{p})-L^{1/2}(\vec{p})\big|$.
\end{theorem}
\begin{proof}
Since $\vec{p}_{\tau}$ minimizes $L_{\tau}^{1/2}$ on $P_{\mathrm{o}}$,
$$L_{\tau}^{1/2}(\vec{p}_{\tau}) \le L_{\tau}^{1/2}(\vec{p}_{\mathrm{o}}).$$
By definition of $\delta_{\tau}$, for any $\vec{p}\in P_{\mathrm{o}}$,
$$\big|L_{\tau}^{1/2}(\vec{p})-L^{1/2}(\vec{p})\big|\le \delta_{\tau}.$$
Hence,
\begin{align*}
L^{1/2}(\vec{p}_{\tau})
&\le \big|L^{1/2}(\vec{p}_{\tau})-L_{\tau}^{1/2}(\vec{p}_{\tau})\big| + L_{\tau}^{1/2}(\vec{p}_{\tau}) \\
&\le \delta_{\tau} + L_{\tau}^{1/2}(\vec{p}_{\mathrm{o}}) \\
&\le \delta_{\tau} + \big|L_{\tau}^{1/2}(\vec{p}_{\mathrm{o}})-L^{1/2}(\vec{p}_{\mathrm{o}})\big| + L^{1/2}(\vec{p}_{\mathrm{o}}) \\
&\le 2\delta_{\tau} + L^{1/2}(\vec{p}_{\mathrm{o}}).
\end{align*}
Since $\vec{p}_{\mathrm{o}}$ minimizes $L^{1/2}$, we also have $L^{1/2}(\vec{p}_{\mathrm{o}})\le L^{1/2}(\vec{p}_{\tau})$. Therefore,
$$0\le L^{1/2}(\vec{p}_{\tau})-L^{1/2}(\vec{p}_{\mathrm{o}})\le 2\delta_{\tau}.$$
Because $\delta_{\tau}\to0$ as $\tau\to0^+$ (by Corollary~\ref{cor:ytoL}), this proves (i).

Now assume Assumption~\ref{assumptionB3}. 
Since $L^{1/2}(\vec{p}_{\tau})-L^{1/2}(\vec{p}_{\mathrm{o}}) \ge c\|\vec{p}_{\tau}-\vec{p}_{\mathrm{o}}\|_2^2,$
we have
$$\|\vec{p}_{\tau}-\vec{p}_{\mathrm{o}}\|_2
\leq c^{-1/2}\big(L^{1/2}(\vec{p}_{\tau})-L^{1/2}(\vec{p}_{\mathrm{o}})\big)
\leq \sqrt{\frac{2}{c}}\delta_{\tau}^{1/2}.$$
Thus $\vec{p}_{\tau}\to\vec{p}_{\mathrm{o}}$ as $\tau\to0^+$, with the rate $\delta_{\tau}^{1/2}$.
\end{proof}

\newpage
\section{Error Estimate}
The previous section describes the limiting behavior of the ideal stationary branch $\{\vec{p}_{\tau}\}$ in the kernel-based smoothing framework for each smoothing level $\tau>0$. While that analysis assumes access to the exact local minimizer $\vec{p}_{\tau}$ of $L_{\tau}$, in practice the corrector step in DePCoN produces only numerical parameter estimates $\{\hat{p}_{\tau_n}\}_{n=1}^{N}$ obtained by running a finite number of iterations of a gradient-based optimizer. This iteration is terminated once a prescribed stopping criterion is satisfied, as described in \cref{subsec:corr_net},
\begin{equation}\label{stopping_criterion}
    \big\|\nabla_{\vec{p}}L_{\tau_n}(\hat{p}_{\tau_n})\big\|_2 \le \varepsilon_n,
\end{equation}
where $\varepsilon_n>0$ is a user-defined tolerance. As a result, the computed estimate $\hat{p}_{\tau_n}$ may differ from the exact minimizer $\vec{p}_{\tau_n}$, leading to a corrector error~\eqref{corrector_error}
\begin{equation}\label{corrector_error}
\vec{e}_{\tau_n} := \hat{p}_{\tau_n} - \vec{p}_{\tau_n}.
\end{equation}
at level $\tau_n$. In addition, even if the exact stationary point $\vec{p}_{\tau_n}$ were provided as input to the predictor, the learned network may not map it exactly to the next stationary point $\vec{p}_{\tau_{n-1}}$. We quantify this intrinsic predictor mismatch $\vec\beta_{\tau_n}$~\eqref{intrinsic_error} along the stationary branch by
\begin{equation}\label{intrinsic_error}
\vec\beta_{\tau_n} := f_\theta(\vec{p}_{\tau_n}) - \vec{p}_{\tau_{n-1}}.
\end{equation}
This term reflects the approximation accuracy of the predictor network with respect to the ideal one-step continuation along the stationary branch.

The predictor step in DePCoN propagates parameter estimates across smoothing levels from $\tau_{n}$ to $\tau_{n-1}$ (see \cref{subsec:pred_net}). Since this propagation is performed over a finite step size $\Delta\tau_{n} := \tau_{n} - \tau_{n-1}$, it introduces an additional $\tau$-discretization error. In the following, we analyze how these two sources of error accumulate across smoothing levels in terms of estimation accuracy.

The error analysis is organized into four steps. While the convergence results in \cref{convergence} describe the limiting behavior of the ideal stationary branch $\{\vec{p}_{\tau}\}$, the following steps quantify how closely the numerical estimates $\{\hat{p}_{\tau_n}\}$ track this branch under finite optimization accuracy and finite $\tau$-discretization.

\textbf{Step C.1 (Differentiability of $L_\tau$ for $\tau>0$).}
Under Assumption~\ref{assumptionB2} and Lemma~\ref{lem:mollify_Lp} (smoothness of $S_{\tau}$), for each fixed $\tau>0$, the trajectory map $\vec{p}\mapsto \vec{y}_{\tau}(\cdot;\vec{p})$ is continuously differentiable. Consequently, the loss $L_{\tau}(\vec{p})$ is continuously differentiable on $P_{\mathrm{o}}$, and $\nabla_{\vec{p}} L_{\tau}$ is well-defined. This justifies the stopping criterion~\eqref{stopping_criterion} and ensures that local Taylor expansions of $\nabla_{\vec{p}}L_\tau$ are valid. (Lemma~\ref{lem:Ltau_diff}.)

\textbf{Step C.2 (Local nondegeneracy and continuation stability).}
Assume that along the selected stationary branch $\{\vec{p}_{\tau}\}$, the Hessian $\nabla_{\vec{p}}^{2}L_{\tau}(\vec{p}_{\tau})$ is invertible for $\tau>0$ (local non-degeneracy condition). The stationary branch is then locally differentiable with respect to $\tau$ and satisfies a continuation equation of the form
$$
\frac{d}{d\tau}\vec{p}_{\tau}
=
- \big(\nabla_{\vec{p}}^{2}L_{\tau}(\vec{p}_{\tau})\big)^{-1}
\,\partial_{\tau}\nabla_{\vec{p}}L_{\tau}(\vec{p}_{\tau}).
$$
In particular, there exists a constant $C_n>0$ such that
$$
\|\vec{p}_{\tau_n}-\vec{p}_{\tau_{n-1}}\|_2
\le C_n\,|\Delta\tau_n|.
$$
This bound quantifies the intrinsic movement of the exact minimizers across smoothing levels. (Lemma~\ref{lem:branch_exist} and Corollary~\ref{cor:finite_step_difference})

\textbf{Step C.3 (Corrector error bound at each smoothing level).}
Let $\hat{p}_{\tau_n}$ satisfy~\eqref{stopping_criterion}. 
Under the local non-degeneracy condition in Step~C.2, we obtain the local estimate
$$
\|\vec{e}_{\tau_n}\|_2
\le \kappa_n\varepsilon_n,
$$
for sufficiently small $\varepsilon_n$, where $\vec e_{\tau_n}=\hat p_{\tau_n}-\vec p_{\tau_n}$. 
This step quantifies the deviation induced by finite optimization accuracy at level $\tau_n$. (Lemma~\ref{lem:corrector_bound_tau_n})

\textbf{Step C.4 (Predictor--corrector one-step error recursion).}
Combining the corrector bound from Step~C.3 with the intrinsic predictor mismatch $\vec\beta_{\tau_n}$, we obtain the one-step tracking estimate 
$$ \|\vec e_{\tau_{n-1}}\|_2 \le \kappa_{n-1}\varepsilon_{n-1} + L_n\|\vec e_{\tau_n}\|_2 + \|\vec\beta_{\tau_n}\|_2, $$
where
$$ \kappa_{n-1} := \left\| \big(\nabla_{\vec p}^2L_{\tau_{n-1}}(\vec p_{\tau_{n-1}}+\xi_{n-1}\vec e_{\tau_{n-1}})\big)^{-1} \right\|_2, \qquad \vec\beta_{\tau_n} := f_\theta(\vec p_{\tau_n})-\vec p_{\tau_{n-1}}. $$
This recursion describes how the corrector residual $\varepsilon_{n-1}$, the propagated error $\vec e_{\tau_n}$, and the predictor mismatch $\vec\beta_{\tau_n}$ jointly determine the tracking accuracy at the next smoothing level. (Theorem~\ref{thm:pc_error_recursion})

\begin{lemma}\label{lem:Ltau_diff} 
The loss
$L_{\tau}(\vec{p})
=
\int_{0}^{T}
\|\vec{y}_{\tau}(t;\vec{p})-\vec{y}_{\mathrm{o}}(t)\|_2^2
dt$ in~\eqref{L_tau_cont}
is continuously differentiable on $P_{\mathrm{o}}$, and 
$\nabla_{\vec{p}}L_{\tau}(\vec{p})$ is well-defined.
\end{lemma}

\begin{proof}
Since $F$ is $C^{1}$ in $(\vec y,\vec p)$ (Assumption~\ref{assumptionB2}), the parameter sensitivity
$$Z_{\tau}(t;\vec p):=\partial_{\vec p}\vec y_{\tau}(t;\vec p)\in\mathbb{R}^{d_y\times d_p}
$$
exists and is characterized as the unique solution of the sensitivity equation
\begin{align}
\frac{d}{dt}Z_{\tau}(t;\vec p)
 &= 
\partial_{\vec y}F\!\big(\vec y_{\tau}(t;\vec p),\vec p,S_{\tau}(t)\big)\,Z_{\tau}(t;\vec p)
+\partial_{\vec p}F\!\big(\vec y_{\tau}(t;\vec p),\vec p,S_{\tau}(t)\big),\nonumber\\
Z_{\tau}(0;\vec p) &= \vec{0}.\label{eq:variational_Z}
\end{align}
Here $\partial_{\vec y}F(\cdot)$ and $\partial_{\vec p}F(\cdot)$ denote the Jacobians of $F$ with respect to $\vec y$ and $\vec p$, respectively.
Under Assumption~\ref{assumptionB2} and Lemma~\ref{lem:mollify_Lp}, the coefficients in~\eqref{eq:variational_Z} are continuous in $t$ along the trajectory $\vec y_{\tau}(t;\vec p)$, so~\eqref{eq:variational_Z} is a linear ODE with a unique solution on $[0,T]$. Hence the solution map $\vec p\mapsto \vec y_{\tau}(\cdot;\vec p)$ is continuously differentiable; see, e.g., \cite{khalil2002nonlinear}.

Consequently, for each fixed $\tau>0$, the loss $L_{\tau}(\vec p)
=
\int_{0}^{T}\|\vec y_{\tau}(t;\vec p)-\vec y_{\mathrm{o}}(t)\|_2^{2}\,dt$ is continuously differentiable with respect to $\vec p$ on $P_{\mathrm{o}}$. Then, by applying chain rule,
\begin{equation}\label{eq:grad_Ltau}
\nabla_{\vec p}L_{\tau}(\vec p)
=
2\int_{0}^{T}
Z_{\tau}(t;\vec p)^{\top}\big(\vec y_{\tau}(t;\vec p)-\vec y_{\mathrm{o}}(t)\big)\,dt.
\end{equation}
Since $Z_{\tau}(t;\vec p)$ is the solution of~\eqref{eq:variational_Z}, $\nabla_{\vec p}L_{\tau}(\vec p)$ is well-defined for all $\vec p\in P_{\mathrm{o}}$ and $\tau>0$.
\end{proof}

\begin{lemma}\label{lem:branch_exist}
Under Assumption~\ref{assumptionB1}, let $\{\vec p_\tau\}_{\tau\in(0,\tau_*]}$ be a family of stationary points satisfying $\nabla_{\vec p} L_\tau(\vec p_\tau)=\vec 0$, and assume that the Hessian $\nabla_{\vec p}^2 L_\tau(\vec p_\tau)$ exists and is invertible for all $\tau\in(0,\tau_*]$. Then the mapping $\tau \mapsto \vec p_\tau$ is locally continuously differentiable on $(0,\tau_*]$.  Moreover, $\vec p_\tau$ satisfies the continuation equation
$$\frac{d}{d\tau}\vec p_\tau = - \big(\nabla_{\vec p}^2 L_\tau(\vec p_\tau)\big)^{-1} \, \partial_\tau \nabla_{\vec p} L_\tau(\vec p_\tau).$$
\end{lemma}

\begin{proof}
Define $G(\vec p,\tau) := \nabla_{\vec p} L_\tau(\vec p).$ By Lemma~\ref{lem:Ltau_diff} $L_\tau$ is continuously differentiable in $\vec p$, and hence
$G$ is continuously differentiable in $(\vec p,\tau)$ for $\tau>0$. Since $G(\vec p_\tau,\tau)=\vec 0$ and $\nabla_{\vec p} G(\vec p_\tau,\tau) = \nabla_{\vec p}^2 L_\tau(\vec p_\tau)$, the implicit function theorem implies 
$\vec p_\tau$ depends differentiably on $\tau$. That is, differentiating the identity
$G(\vec p_\tau,\tau)=\vec 0$
with respect to $\tau$ yields
$$\nabla_{\vec p}^2 L_\tau(\vec p_\tau) \frac{d}{d\tau}\vec p_\tau + \partial_\tau \nabla_{\vec p} L_\tau(\vec p_\tau) = \vec 0,$$
from which the continuation equation follows.
\end{proof}

\begin{corollary}\label{cor:finite_step_difference}
Under the setting of Lemma~\ref{lem:branch_exist}, the stationary branch $\tau\mapsto \vec p_\tau$ is locally Lipschitz on $(0,\tau_*)$. In particular, for any consecutive levels $\tau_{n-1}<\tau_n<\tau_*$, there exists $\xi_n\in(\tau_{n-1},\tau_n)$ such that
$$\vec p_{\tau_n}-\vec p_{\tau_{n-1}}=\frac{d}{d\tau}\vec p_{\xi_n}\,(\tau_n-\tau_{n-1}).$$
Consequently,
$$\|\vec p_{\tau_n}-\vec p_{\tau_{n-1}}\|_2\le C_n\,|\Delta\tau_n|,$$
where $\Delta\tau_n:=\tau_n-\tau_{n-1}$ and $C_n:=\sup_{\tau\in[\tau_{n-1},\tau_n]}\big\|\frac{d}{d\tau}\vec p_\tau\big\|_2<\infty$.
\end{corollary}



\begin{lemma}\label{lem:corrector_bound_tau_n}
Fix $n\in\{0,\dots,N\}$ and let $\vec p_{\tau_n}$ be the selected stationary point on the branch in Lemma~\ref{lem:branch_exist} and let $\hat p_{\tau_n}$ be a parameter estimate from DePCoN. Define the error
$\vec e_{\tau_n}:=\hat p_{\tau_n}-\vec p_{\tau_n}$. Assume that $L_{\tau_n}$ is three times continuously differentiable in a neighborhood of $\vec p_{\tau_n}$.

Then there exists $\xi_{\tau_n}\in(0,1)$ such that
$$ \nabla_{\vec p}L_{\tau_n}(\hat p_{\tau_n}) = \nabla_{\vec p}^2L_{\tau_n} \big(\vec p_{\tau_n}+\xi_{\tau_n}\vec e_{\tau_n}\big) \,\vec e_{\tau_n}. $$
If the Hessian is invertible at $\vec p_{\tau_n}+\xi_{\tau_n}\vec e_{\tau_n}$, then
$$ \vec e_{\tau_n} = \nabla_{\vec p}^2L_{\tau_n} (\vec p_{\tau_n}+\xi_{\tau_n}\vec e_{\tau_n})^{-1} \nabla_{\vec p}L_{\tau_n}(\hat p_{\tau_n}),$$
and hence
$$ \|\vec e_{\tau_n}\|_2 \le \kappa_n\,\varepsilon_n, $$
where $\kappa_n := \left\| \nabla_{\vec p}^2L_{\tau_n} (\vec p_{\tau_n}+\xi_{\tau_n}\vec e_{\tau_n})^{-1} \right\|_2$.
\end{lemma}

\begin{proof}
Since $\vec p_\tau$ is stationary, $G(\vec p_\tau)=\vec 0$.
Apply the (vector-valued) mean value theorem to $G$ along the segment
$\vec p_\tau + s\vec e_\tau$, $s\in[0,1]$:
there exists $\xi_\tau\in(0,1)$ such that
$$
G(\hat p_\tau)-G(\vec p_\tau)
=
\nabla_{\vec p}G(\vec p_\tau+\xi_\tau\vec e_\tau)\,\vec e_\tau.
$$
Recalling that $\nabla_{\vec p}G(\vec p)=\nabla_{\vec p}^2L_\tau(\vec p)$, we obtain
$$
\nabla_{\vec p}L_\tau(\hat p_\tau)
=
\nabla_{\vec p}^2L_\tau(\vec p_\tau+\xi_\tau\vec e_\tau)\,\vec e_\tau,
$$
which proves the claimed identity.

If $\nabla_{\vec p}^2L_\tau(\vec p_\tau+\xi_\tau\vec e_\tau)$ is invertible, then we can solve for $\vec e_\tau$:
$$
\vec e_\tau
=
\Big(
\nabla_{\vec p}^2L_\tau(\vec p_\tau+\xi_\tau\vec e_\tau)
\Big)^{-1}
\nabla_{\vec p}L_\tau(\hat p_\tau).
$$
Taking the Euclidean operator norm gives
$$
\|\vec e_\tau\|_2
\le
\left\|
\Big(
\nabla_{\vec p}^2L_\tau(\vec p_\tau+\xi_\tau\vec e_\tau)
\Big)^{-1}
\right\|_2
\,
\|\nabla_{\vec p}L_\tau(\hat p_\tau)\|_2.
$$
Finally, if $\|\nabla_{\vec p}L_\tau(\hat p_\tau)\|_2\le \varepsilon_n$, the stated bound
$\|\vec e_{\tau_n}\|_2\le \kappa_n\,\varepsilon_n$
follows.
\end{proof}



\paragraph*{Remark.} Since the predictor $f_\theta$ is a feed-forward network composed of affine maps and ReLU activations (see \cref{subsec:pred_net}), it is piecewise affine and hence Lipschitz continuous on any compact set. In particular, on the compact admissible set $P_{\mathrm{o}}$, there exists a constant $L\ge0$ such that
$$\|f_\theta(\vec p)-f_\theta(\vec q)\|_2\le L\|\vec p-\vec q\|_2,\qquad \forall\,\vec p,\vec q\in P_{\mathrm{o}}.$$

Under this property, we have the following theorem.
\begin{theorem}\label{thm:pc_error_recursion}
Fix $n\in\{1,\dots,N\}$. Assume the selected stationary branch $\{\vec p_{\tau_k}\}_{k=0}^N$ is well-defined, and suppose the corrector output $\hat p_{\tau_{n-1}}$ satisfies the stopping criterion
$$\|\nabla_{\vec p}L_{\tau_{n-1}}(\hat p_{\tau_{n-1}})\|_2\le \varepsilon_{n-1}.$$
Assume further that the Hessian is invertible along the line segment between $\vec p_{\tau_{n-1}}$ and $\hat p_{\tau_{n-1}}$, i.e.,
$$\nabla_{\vec p}^2L_{\tau_{n-1}}(\vec p_{\tau_{n-1}}+s\,\vec e_{\tau_{n-1}})\ \text{is invertible for all } s\in[0,1],$$
where $\vec e_{\tau_{n-1}}:=\hat p_{\tau_{n-1}}-\vec p_{\tau_{n-1}}$.

Then the tracking error satisfies the one-step bound
$$\|\vec e_{\tau_{n-1}}\|_2\le \kappa_{n-1}\varepsilon_{n-1}+L_n\|\vec e_{\tau_n}\|_2+\| \vec\beta_{\tau_n}\|_2,$$
where $$ \kappa_{n-1} := \left\| \big(\nabla_{\vec p}^2L_{\tau_{n-1}}(\vec p_{\tau_{n-1}}+\xi_{n-1}\vec e_{\tau_{n-1}})\big)^{-1} \right\|_2, \qquad \vec\beta_n := f_\theta(\vec p_{\tau_n})-\vec p_{\tau_{n-1}}.$$
\end{theorem}

\begin{proof}
We start from the predictor recursion
$$ \hat p_{\tau_{n-1}} = f_\theta(\hat p_{\tau_n}), \qquad \vec e_{\tau_k}:=\hat p_{\tau_k}-\vec p_{\tau_k}. $$
Insert and subtract $f_\theta(\vec p_{\tau_n})$ to decompose the error at level $\tau_{n-1}$:
\begin{align*}
\vec e_{\tau_{n-1}}
&= \hat p_{\tau_{n-1}}-\vec p_{\tau_{n-1}}
= f_\theta(\hat p_{\tau_n})-\vec p_{\tau_{n-1}} \\
&= \big(f_\theta(\hat p_{\tau_n})-f_\theta(\vec p_{\tau_n})\big)
+ \big(f_\theta(\vec p_{\tau_n})-\vec p_{\tau_{n-1}}\big) \\
&= \big(f_\theta(\hat p_{\tau_n})-f_\theta(\vec p_{\tau_n})\big) + \vec\beta_n .
\end{align*}
Taking norms and using the assumed Lipschitz continuity of $f_\theta$ on $P_{\mathrm{o}}$
(with Lipschitz constant $L_n$ at this step) yields
$$ \|\vec e_{\tau_{n-1}}\|_2 \le L_n\|\hat p_{\tau_n}-\vec p_{\tau_n}\|_2 + \|\vec\beta_n\|_2 = L_n\|\vec e_{\tau_n}\|_2 + \|\vec\beta_n\|_2. $$
It remains to incorporate the contribution of the corrector stopping criterion at level $\tau_{n-1}$.

Since $\vec p_{\tau_{n-1}}$ lies on the selected stationary branch, it satisfies
$$\nabla_{\vec p}L_{\tau_{n-1}}(\vec p_{\tau_{n-1}})=\vec 0.$$
Apply the mean value theorem to the vector-valued map
$\nabla_{\vec p}L_{\tau_{n-1}}(\cdot)$ along the segment
$\vec p_{\tau_{n-1}} + s\,\vec e_{\tau_{n-1}}$ for $s\in[0,1]$:
there exists $\xi_{n-1}\in(0,1)$ such that
\begin{align*}
\nabla_{\vec p}L_{\tau_{n-1}}(\hat p_{\tau_{n-1}}) &= \nabla_{\vec p}L_{\tau_{n-1}}(\vec p_{\tau_{n-1}}) + \nabla_{\vec p}^2L_{\tau_{n-1}}(\vec p_{\tau_{n-1}}+\xi_{n-1}\vec e_{\tau_{n-1}})\,\vec e_{\tau_{n-1}} \\
&= \nabla_{\vec p}^2L_{\tau_{n-1}}(\vec p_{\tau_{n-1}}+\xi_{n-1}\vec e_{\tau_{n-1}})\,\vec e_{\tau_{n-1}}.
\end{align*}
By the assumed invertibility of the Hessian along the segment, we may solve for $\vec e_{\tau_{n-1}}$:
$$ \vec e_{\tau_{n-1}} = \big(\nabla_{\vec p}^2L_{\tau_{n-1}}(\vec p_{\tau_{n-1}}+\xi_{n-1}\vec e_{\tau_{n-1}})\big)^{-1} \nabla_{\vec p}L_{\tau_{n-1}}(\hat p_{\tau_{n-1}}). $$
Taking norms and using the stopping criterion
$\|\nabla_{\vec p}L_{\tau_{n-1}}(\hat p_{\tau_{n-1}})\|_2\le \varepsilon_{n-1}$
gives
$$ \|\vec e_{\tau_{n-1}}\|_2 \le \kappa_{n-1}\,\varepsilon_{n-1}, \qquad \kappa_{n-1} := \left\| \big(\nabla_{\vec p}^2L_{\tau_{n-1}}(\vec p_{\tau_{n-1}}+\xi_{n-1}\vec e_{\tau_{n-1}})\big)^{-1} \right\|_2. $$

Finally, combine the two bounds:
\begin{align*}
\|\vec e_{\tau_{n-1}}\|_2 &\le \kappa_{n-1}\varepsilon_{n-1} + \|\vec e_{\tau_{n-1}}\|_2 \\
&\le \kappa_{n-1}\varepsilon_{n-1} + L_n\|\vec e_{\tau_n}\|_2 + \|\vec\beta_n\|_2,
\end{align*}
which is the claimed one-step estimate.
\end{proof}

\newpage
\section{Supplementary tables and figures}

\begin{table*}[!ht]
\caption{Various applications of non-autonomous dynamical systems driven by discontinuous exogenous inputs.
Such abrupt or event-driven exogenous inputs frequently arise in real-world settings and pose significant challenges for learning and optimization in dynamical systems.}
\label{tab:nonauto_appl_full}
\resizebox{\linewidth}{!}{%
\begin{tabular}{cccc}
\toprule
Category & Application & Discontinuous External Input & Reference \\
\midrule
\multirow{3}{*}{1. Control Theory} & Traffic Signal Control & \begin{tabular}[c]{@{}c@{}}Periodic switching of traffic lights (Red/Green),\\ discrete gating of vehicle flows.\end{tabular} & \begin{tabular}[c]{@{}c@{}}\cite{33aboudolas2009store},\\ \cite{42du2014positive}\end{tabular} \\ 
 & \begin{tabular}[c]{@{}c@{}}Demand Response \\ \& Smart Thermostats\end{tabular} & \begin{tabular}[c]{@{}c@{}}Real-time price signals,\\ grid load shedding\end{tabular} & \begin{tabular}[c]{@{}c@{}}\cite{34mathieu2012using},\\ \cite{36chassin2015new}\end{tabular} \\ \cmidrule(rl){1-4}
\multirow{12}{*}{2. Life \& Ecological sciences} & Pharmacokinetics & \begin{tabular}[c]{@{}c@{}}Bolus injections, insulin pump delivery,\\ Michaelis-Menten elimination\end{tabular} & \begin{tabular}[c]{@{}c@{}}\cite{15huang2012modeling},\\ \cite{29song2014modeling},\\ \cite{30tang2007one}\end{tabular} \\
 & Bioprocess Engineering & \begin{tabular}[c]{@{}c@{}}Impulsive feeding,\\ fed-batch fermentation control\end{tabular} & \begin{tabular}[c]{@{}c@{}}\cite{23shen2016nonlinear},\\ \cite{24liu2022robust},\\ \cite{25gao2006nonlinear}\end{tabular} \\
 & Ecological Systems & \begin{tabular}[c]{@{}c@{}}Impulsive ecological interventions\\ (harvesting, pest control, nutrient pulses)\end{tabular} & \begin{tabular}[c]{@{}c@{}}\cite{4bainov1990population},\\ \cite{6meng2010dynamic},\\ \cite{7nie2009existence},\\ \cite{8nie2009dynamics}\end{tabular} \\
 & NF-$\kappa$B Signaling Pathways & External stimuli (e.g. TNF-$\alpha$) & \cite{41nelson2004oscillations} \\
 & Circadian Clock & \begin{tabular}[c]{@{}c@{}}Light/Dark pulses, and discrete data\\ sampling intervals from wearable devices\end{tabular} & \begin{tabular}[c]{@{}c@{}}\cite{39forger1999simpler},\\ \cite{44lim2025enhanced}\end{tabular} \\ \cmidrule(rl){1-4}
\multirow{10}{*}{3. Physics \& Engineering} & Aerospace & \begin{tabular}[c]{@{}c@{}}Instantaneous thrust,\\ impulsive maneuvers\end{tabular} & \begin{tabular}[c]{@{}c@{}}\cite{12carter1991optimal},\\ \cite{13carter2000necessary}\end{tabular} \\
 & Mechanics & \begin{tabular}[c]{@{}c@{}}Vibro-impact (collisions),\\ clock escapement energy, dry friction\end{tabular} & \begin{tabular}[c]{@{}c@{}}\cite{16brogliato1996nonsmooth},\\ \cite{26ibrahim2009vibro},\\ \cite{27moline2012model}\end{tabular} \\
 & Bipedal walking robot & Impact forces at foot–ground contact & \cite{37westervelt2003hybrid} \\ 
 & Waveguides & \begin{tabular}[c]{@{}c@{}}Structural discontinuities (conical/elliptical)\\ in THz waveguides\end{tabular} & \cite{28heydari2019analysis} \\
 & DC-DC Converters & \begin{tabular}[c]{@{}c@{}}High-frequency On/Off switching\\ of transistors\end{tabular} & \begin{tabular}[c]{@{}c@{}}\cite{43papafotiou2004hybrid},\\ \cite{45siddhartha2018non}\end{tabular} \\ 
 & Power Electronics & Switching events, component faults & \cite{22dominguez2010detection} \\ \cmidrule(rl){1-4}
\multirow{4}{*}{4. Economics \& Finance} & Market Dynamics & Demand–supply shocks & \begin{tabular}[c]{@{}c@{}}\cite{11stamova2004stability},\\ \cite{20stamov2011almost}\end{tabular} \\
 & Financial Control & \begin{tabular}[c]{@{}c@{}}Central bank interventions,\\ interest rate adjustments, transaction costs\end{tabular} & \begin{tabular}[c]{@{}c@{}}\cite{19cadenillas2000classical},\\ \cite{21bielecki2000risk},\\ \cite{38clarida2000monetary}\end{tabular} \\ \cmidrule(rl){1-4}
5. Neural Networks & Neural Dynamics & Spiking synaptic activity & \cite{10stamov2007almost} \\ \cmidrule(rl){1-4}
6. Epidemiology & \begin{tabular}[c]{@{}c@{}}Non-pharmaceutical\\Interventions (NPIs)\end{tabular} & Discrete policy shocks & \cite{40flaxman2020estimating} \\
\bottomrule
\end{tabular}%
}
\end{table*}
\newpage
\begin{figure}[t!]
  \begin{center}
    \centerline{\includegraphics[width=\linewidth]{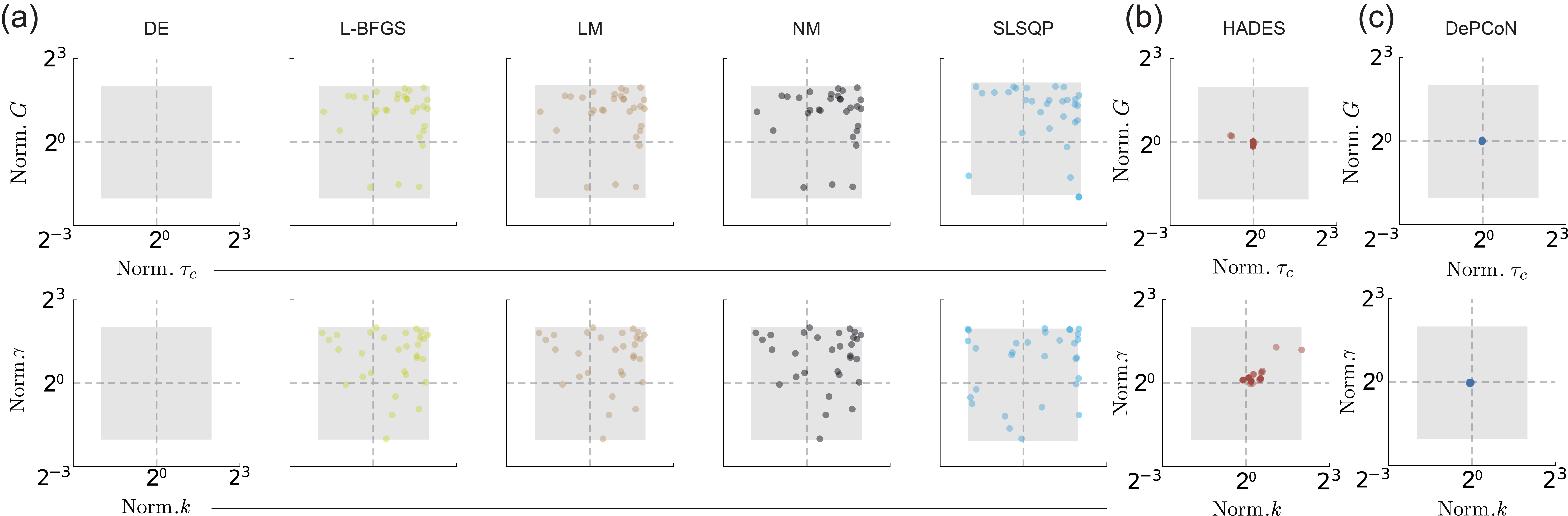}}
    \caption{Comparison of parameter estimation methods on the human circadian pacemaker model introduced in~\cref{fig:1}.
The scatter plots summarize parameter estimates from 30 independent trials for each of seven methods
(\emph{DE, L-BFGS, LM, NM, SLSQP, HADES-NN, and DePCoN}). All parameter estimates are normalized
by the true parameter values, which are located at the intersection of the dashed lines, and gray boxes indicate the parameter ranges used for random initialization.
DE failed to produce valid parameter estimates due to numerical divergence during the search process.
(a) Under discontinuous exogenous inputs, the resulting non-smooth loss landscape causes \emph{L-BFGS}, \emph{LM}, and \emph{NM} to remain trapped near their initializations.
(b) HADES-NN achieves improved estimation performance relative to these baselines,
whereas (c) DePCoN consistently and accurately recovers the ground-truth parameter values across trials. 
}
    \label{fig:5}
  \end{center}
\end{figure}
\begin{figure}[h!]
  \begin{center}
    \centerline{\includegraphics[width=\linewidth]{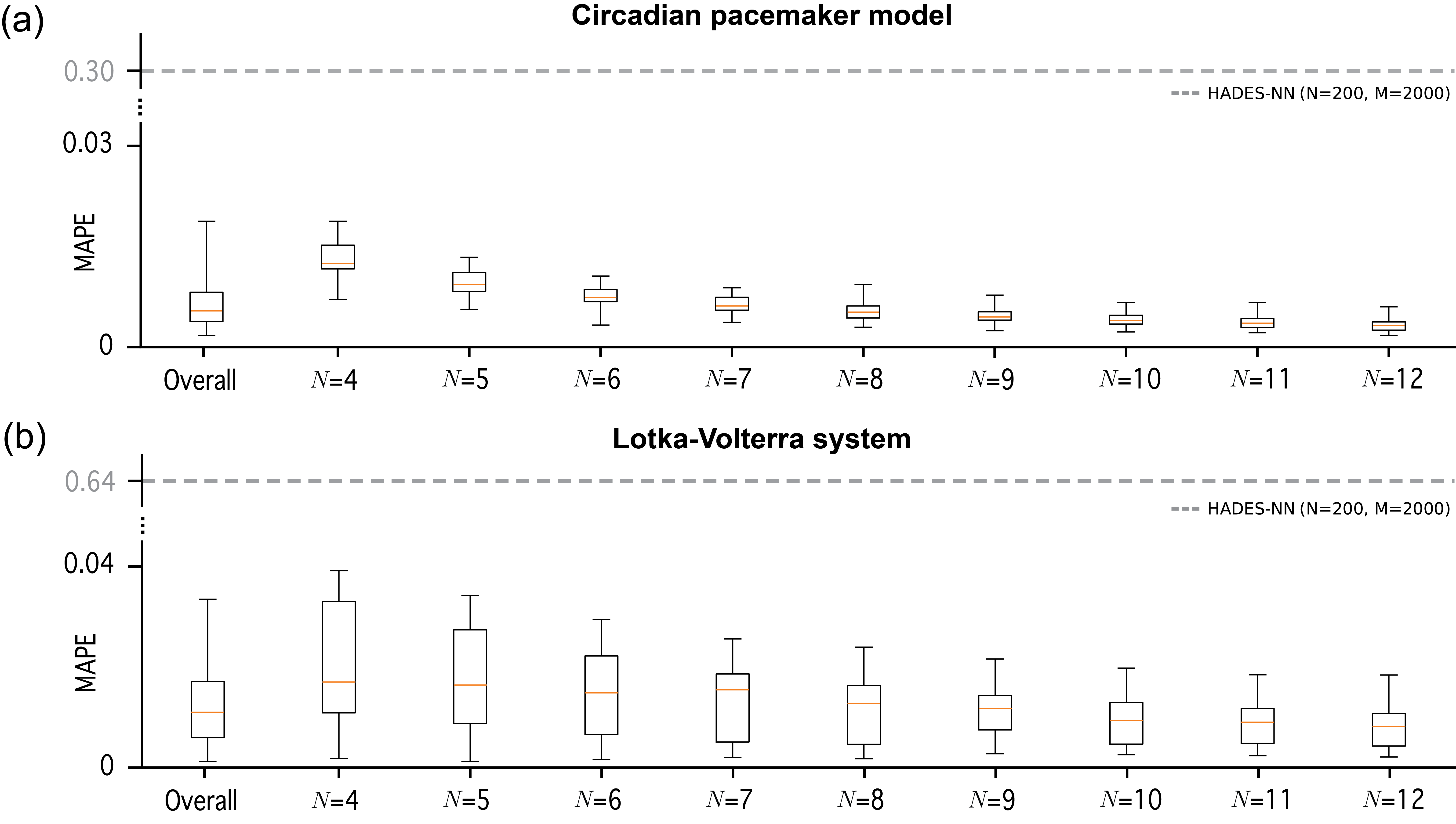}}
    \caption{Sensitivity of DePCoN to the smoothing-grid size $N$.
We evaluate the robustness of DePCoN with respect to the number of smoothing scales $N$ using box plots of the mean absolute percentage error (MAPE).
\textbf{(a)} Circadian pacemaker model~\cref{fig:1}-(a).
\textbf{(b)} Lotka--Volterra system~\cref{fig:4}-(a).
For each $N \in \{4,\ldots,12\}$, we perform 30 independent trials with random initial parameters sampled from $[p_{\mathrm{o},i}/4,\,4p_{\mathrm{o},i}]$ for $i=1,\ldots,4$.
In each trial, the MAPE of the estimated parameters is computed, and its distribution is summarized as a box plot for each $N$.
The leftmost ``Overall'' box plot aggregates MAPE values across all smoothing-grid sizes as a reference.
The horizontal dashed line indicates the MAPE level achieved by \emph{HADES-NN} under its reported hyperparameter setting, providing a baseline for comparison.
Across both systems, DePCoN maintains stable MAPE distributions over a wide range of $N$, consistently matching or improving upon the HADES-NN baseline, which demonstrates strong robustness to variations in the smoothing-grid hyperparameter.}
    \label{fig:6}
  \end{center}
\end{figure}
\end{document}